\documentclass[12pt]{article}

\usepackage{amssymb}
\usepackage{amsmath,amsfonts}
\usepackage{amsthm,bbm,eufrak,upgreek}
\usepackage{mathrsfs,booktabs}
\usepackage{amsxtra,mathtools}
\usepackage[colorlinks=True,linkcolor=blue,anchorcolor=blue,citecolor=blue,filecolor=blue,CJKbookmarks=True]{hyperref}
\usepackage[a4paper,left=2.6cm,right=2.6cm,top=3.2cm,bottom=3.2cm]{geometry}
\usepackage[hang,flushmargin]{footmisc}

\newcommand{\1}{{\mathbbm 1}}
\newcommand{\C}{{\mathbb C}}
\newcommand{\Z}{{\mathbb Z}}

\newcommand{\la}{\langle}
\newcommand{\ra}{\rangle}

\DeclareMathOperator{\Aut}{Aut}

\DeclareMathOperator{\End}{End}

\DeclareMathOperator{\wt}{wt}
\DeclareMathOperator{\Hom}{Hom}

\newtheorem{thm}{Theorem}[section]
\newtheorem{prop}[thm]{Proposition}
\newtheorem{lem}[thm]{Lemma}

\newtheorem{rmk}[thm]{Remark}
\newtheorem{defn}[thm]{Definition}

\begin{document}

\begin{center}
{\Large \bf  Fusion rules for the orbifold vertex operator algebras $L_{\widehat{\frak{sl}_2}}(k,0)^{K}$  }
\end{center}

\begin{center}
{Rabia Iqbal
\footnote{Supported by China NSF grants No. 12171312. Email: rabipu14@sjtu.edu.cn.}
and Cuipo Jiang
\footnote{Supported by China NSF grant No.12171312.  Email: cpjiang@sjtu.edu.cn.}\\
School of Mathematical  Sciences, Shanghai Jiao Tong University\\
Shanghai 200240, China}
\end{center}

\let\thefootnote\relax

\begin{abstract}
For the Klein group $K$  and a positive integr $k$, irreducible modules of the orbifold vertex operator algebra $L_{\widehat{\mathfrak{sl}_2}}(k,0)^{K}$  have been  classified and constructed in \cite{JWa}. In this paper, we determine completely the fusion rules of 
$L_{\widehat{\frak{sl}_2}}(k,0)^{K}$.
\end{abstract}


\section{Introduction}

 Let $\mathfrak{sl}_2(\C)$ be the 3-dimensional simple Lie algebra over the complex field $\C$. For a positive integer $k$, let    $L_{\widehat{\mathfrak{sl}_2}}(k,0)$ be the  associated rational and simple affine vertex operator algebra  \cite{FZ92}, \cite{LL04}. It is known that the automorphism group of 
 $L_{\widehat{\mathfrak{sl}_2}}(k,0)$ is $PSL(2)$. Given a finite subgroup $G$ of $PSL(2)$, the orbifold vertex operator algebra $L_{\widehat{\mathfrak{sl}_2}}(k,0)^G$ is simple, and  is conjecturely rational which is shown to be true when $G$ is solvable \cite{CM}. It is very natrual and desirable to study the structure and representations of the orbifold vertex operator algebra $L_{\widehat{\mathfrak{sl}_2}}(k,0)^G$.
 If $|G|=2$, the irreducible modules and fusion rules of $L_{\widehat{\mathfrak{sl}_2}}(k,0)^G$ are characterized in \cite{JW2}.  If  $G$ is  the Klein subgroup $K$ of $\Aut(L_{\widehat{\mathfrak{sl}_2}}(k,0))$ generated by involutions $\sigma_1$ and $\sigma_2$, which are defined by
 $\sigma_1(h(-1)\mathbbm{1})=h(-1)\mathbbm{1}$, $\sigma_1(e(-1)\mathbbm{1})=-e(-1)\mathbbm{1}$, $\sigma_1(f(-1)\mathbbm{1})=-f(-1)\mathbbm{1}$, and 
 $\sigma_2(h(-1)\mathbbm{1})=-h(-1)\mathbbm{1}$, $\sigma_2(e(-1)\mathbbm{1})=f(-1)\mathbbm{1}$, $\sigma_2(f(-1)\mathbbm{1})=e(-1)\mathbbm{1}$, respectively,  
  the irreducible modules of $L_{\widehat{\mathfrak{sl}_2}}(k,0)^K$ are classified in \cite{JWa}, where $\{e,f,h\}$ is a standard basis of $\mathfrak{sl}_2(\C)$.  In this paper we determine  the fuison rules completely for $L_{\widehat{\mathfrak{sl}_2}}(k,0)^K$. As pointed out in \cite{JWa}, if $k=1$,  $L_{\widehat{\mathfrak{sl}_2}}(1,0)^K$ is isomorphic to the orbifold lattice vertex algebra $ V_{\Z\beta}^+$ with  $(\beta|\beta)=8$. Classification of irreducible modules and  fusion rules of $V_{\Z\gamma}^+$  with $(\gamma|\gamma)\in 2\Z_+$ were given in \cite{DN99} and \cite{Ab01} respectively. If $k=2$, then $L_{\widehat{\mathfrak{sl}_2}}(2,0)^K\cong L(\frac{1}{2},0)^{\otimes 3}$, where $L(\frac{1}{2},0)$ is the simple Virasoro vertex operator algebra with central charge $\frac{1}{2}$.  The  fusion rules of $L(\frac{1}{2},0)$ were given in \cite{DMZ94} and \cite{W93}. So we assume $k\geq 3$ in this paper.
  Notice that the Klein group $K$ is not  a cyclic group. This means that we need to determine fusion products of modules of $L_{\widehat{\mathfrak{sl}_2}}(k,0)^K$ coming from twisted modules of $L_{\widehat{\mathfrak{sl}_2}}(k,0)$ with respect to different involutions. However, the associated intertwining operators are usually not easy to construct  in such cases. 
  Instead of constructing the associated intertwining operators directly, we notice that when $k$ is an  even positive integer, $V_{\Z\gamma}^+\otimes K_0^+$ is a conformal vertex subalgebra of  $L_{\widehat{\mathfrak{sl}_2}}(k,0)^K$.  Then we can decompose the irreducibule modules of 
  $L_{\widehat{\mathfrak{sl}_2}}(k,0)^K$ into direct sum of irreducible modules of  $V_{\Z\gamma}^+\otimes K_0^+$. By using the fusion rules obtained in \cite{Ab01} and \cite{JW2} we get the desired fusion products, where $V_{\Z\gamma}^+$ and $K_0^+$  are  the orbifold subalgebras of   lattice vertex operator algbera  and  parafermion vertex operator algbera with respect to the involution $\sigma_2$ and 
  the Heisenberg Lie algebra ${\C}h$, respectively. The case that $k$ is an odd integer is  easier to deal with. We would like to point out that with the geven fusion rules, the $S$-matrix of  $L_{\widehat{\mathfrak{sl}_2}}(k,0)^K$ can be easily given. 
  
  The study of $L_{\widehat{\mathfrak{sl}_2}}(k,0)^K$ is also partly motivated by the work  \cite{JLam19}. 
Let  $\mathfrak{so}_{m}$ be the orthogonal simple Lie algebra over $\C$, and $L_{\widehat{\mathfrak{so}_m}}(1,0)$  the associated simple and rational affine vertex operator algebra with level 1.
 For $\ell \in \mathbb{Z}_{\geqslant 1}$, the tensor product $L_{\widehat{\mathfrak{so}_m}}(1,0)^{\otimes \ell}$ is still rational and the diagonal action of the affine Lie algebra ${\widehat{\mathfrak{so}_m}}$ on $L_{\widehat{\mathfrak{so}_m}}(1,0)^{\otimes \ell}$ defines a vertex subalgebra $L_{\widehat{\mathfrak{so}_m}}(\ell,0)$ of level $\ell$.
The commutant vertex operator algebras $C_{{L_{\widehat{\mathfrak{so}_m}}(1,0)}^{\otimes \ell}}({L_{\widehat{\mathfrak{so}_m}}(\ell,0)})$ for $m\geqslant 4$, $\ell \geqslant 3$ were characterized in \cite{JLam19} as an orbifold vertex operator algebra,  which is isomorphic to  $L_{\widehat{\mathfrak{so}_l}}(m,0)^G$ or 
$(L_{\widehat{\mathfrak{so}_l}}(m,0)\oplus L_{\widehat{\mathfrak{so}_l}}(m,m\Lambda_1))^G$ for $l\geq 4$, where $G$ is a finite abelian group generated by involutions. When $l=3$,
 $C_{{L_{\widehat{\mathfrak{so}_m}}(1,0)}^{\otimes 3}}({L_{\widehat{\mathfrak{so}_m}}(3,0)})$ can be realized as the obifold vertex operator algebra $L_{\widehat{\mathfrak{sl}_2}}(2m,0)^{K}$ if $m$ is odd,  and  $(L_{\widehat{\mathfrak{sl}_2}}(2m,0)+L_{\widehat{\mathfrak{sl}_2}}(2m,2m))^{K}$ if $m$ is even. Furthermore, $C_{{L_{\widehat{\mathfrak{so}_m}}(1,0)}^{\otimes 3}}({L_{\widehat{\mathfrak{so}_m}}(3,0)})$ is the building block of $C_{{L_{\widehat{\mathfrak{so}_m}}(1,0)}^{\otimes \ell}}({L_{\widehat{\mathfrak{so}_m}}(\ell,0)})$.

The paper is organized as follows. In Section 2, we briefly review some basic notations and facts on vertex operator algebras. In Section 3, we  give  several  lemmas and propositions which are very useful  in our determining the fusion rules.   Section 4 is dedicated to give  all the fusion rules of $L_{\widehat{\mathfrak{sl}_2}}(k,0)^K$.  Throughout the  paper, we always assume that $k$ is a positive integer.

We use the symbols $\C$, $\Z$, and  $\Z_+$  for  the sets  of complex numbers,  integers,  and positive integers, respectively.

\section{Preliminaries}

In this section, we first review some basics and  facts on  vertex operator algebras from \cite{DLM98-1}, \cite{DLM00}, \cite{FHL93}, \cite{LL04} and \cite{Z96}, \cite{DJX13},  \cite{MT04}. Then we recall  some main  results  from \cite{JWa}.

Let $(V,Y,\mathbbm{1},\omega)$ be a vertex operator algebra \cite{Bor86}, \cite{FLM88}.  Let $g$ be an automorphism of the vertex operator algebra $V$ of finite order $T$. Denote the decomposition of $V$ into eigenspaces of $g$ as:
\[ V=\bigoplus_{r\in \mathbb{Z}/T\mathbb{Z}}V^r,\]
where $V^r= \{v \in V |gv = e^{-2\pi \sqrt{-1}\frac{r}{T}}v\}$, $0 \leqslant r \leqslant T-1$. We use $r$ to denote both an integer between $0$ and $T-1$ and its residue class modulo $T$ in this situation.

\begin{defn}

Let $V$ be a vertex operator algebra. A weak $g$-twisted $V$-module is a vector space $M$ equipped with a linear map
\begin{align*}
Y_M(\cdot, x) : V & \longrightarrow (\End M)[[x^{\frac{1}{T}},x^{-\frac{1}{T}}]] \\
                v & \longmapsto Y_M(v,x) = \sum_{n\in \frac{1}{T}\mathbb{Z}}v_nx^{-n-1},
\end{align*}
 where $v_n \in \End M $, satisfying the following conditions for $0 \leqslant r \leqslant T-1$, $u \in V^r, v \in V$, $w \in M$:
\[Y_M(u,x) = \sum_{n\in \frac{r}{T}+\mathbb{Z}}u_nx^{-n-1},\]
\[u_sw=0 \quad\text{for} \quad s \gg 0,\]
\[Y_M({\bf 1},x) = id_M,\]
\[x_{0}^{-1}\delta(\frac{x_1-x_2}{x_0})Y_M(u,x_1)Y_M(v,x_2)-x_{0}^{-1}\delta(\frac{x_2-x_1}{-x_0})Y_M(v,x_2)Y_M(u,x_1)\]
\[=x_{2}^{-1}(\frac{x_1-x_0}{x_2})^{-\frac{r}{T}}\delta(\frac{x_1-x_0}{x_2})Y_M(Y(u,x_0)v,x_2),\]
where $\delta(x)=\sum_{n\in\mathbb{Z}}x^n$ and all binomial expressions are to be expanded in nonnegative integral powers of the second variable.
\end{defn}

We have the following Borcherds identities  \cite{DLM98-1}, \cite{Ab01}.

\begin{equation}   \label{Borcherds identity 1}
[ u_{m + \frac{r}{T}}, v_{n + \frac{s}{T}} ] = \sum_{i=0}^{\infty} \binom{m + \frac{r}{T}}{i}(u_iv)_{m + n + \frac{r+s}{T}-i},
\end{equation}

\begin{equation}    \label{Borcherds identity 2}
\sum_{i=0}^{\infty} \binom{\frac{r}{T}}{i}(u_{m+i}v)_{n+\frac{r+s}{T}-i} = \sum_{i=0}^{\infty} (-1)^i \binom{m}{i} (u_{m+\frac{r}{T}-i}v_{n+\frac{s}{T}+i} -(-1)^mv_{m+n+\frac{s}{T}-i}u_{\frac{r}{T}+i}),
\end{equation}
where $u \in V^r$, $v \in V^s$, $m$, $n \in \mathbb{Z}$.

\begin{defn}
An admissible $g$-twisted $V$-module is a weak $g$-twisted $V$-module which carries a $\frac{1}{T}\mathbb{Z}_{\geqslant 0}$-grading $M = \oplus_{n\in \frac{1}{T}\mathbb{Z}_{\geqslant 0}}M(n)$ satisfying $v_mM(n) \subseteq M(n+r-m-1)$ for homogeneous $v \in V_r$, $m$, $n \in\frac{1}{T}\mathbb{Z}$.
\end{defn}

\begin{defn}
A $g$-twisted $V$-module is a weak $g$-twisted $V$-module which carries a $\mathbb{C}$-grading:
\[M=\oplus_{ \lambda \in \mathbb{C}}M_\lambda,\]
such that dim $M_{\lambda} < \infty$, $M_{\lambda +\frac{n}{T}} = 0$ for fixed $\lambda$ and $n \ll 0$, $L(0)w = \lambda w = ( \wt w )w$ for $w \in M_{\lambda}$, where $L(0)$ is the component operator of $Y_M(\omega, x) = \sum_{n\in \mathbb{Z}}L(n)x^{-n-2}$.
\end{defn}


\begin{defn}
A vertex operator algebra is called $C_2$-cofinite if $V/C_2(V)$ is finite-dimensional,
where $C_2(V) = \langle u_{-2}v \mid u, v \in V \rangle$.
\end{defn}

We have the following result from \cite{ABD04}, \cite{DLM98-1} and \cite{Z96}.

\begin{thm}
Let $V$ be a vertex operator algebra satisfying the $C_2$-cofinite property, then $V$ has only finitely many irreducible admissible modules up to isomorphism. The rationality of $V$ also implies the same result.
\end{thm}

We have the following results from \cite{DLM98-1} and \cite{DLM00}.

\begin{thm}
Let V be a $g$-rational vertex operator algebra, then

(1) Any irreducible admissible $g$-twisted $V$-module $M$ is a $g$-twisted $V$-module. Moreover, there exists a number $\lambda \in \mathbb{C}$ such that $M=\oplus_{n\in \frac{1}{T}\mathbb{Z}_{\geqslant 0}}M_{\lambda + n}$, where $M_{\lambda} \ne 0$. The number $\lambda$ is called the conformal weight of $M$;

(2)There are only finitely many irreducible admissible $g$-twisted $V$-modules up to isomorphism.
\end{thm}
\begin{defn} Let $M=\bigoplus_{n\in\frac{1}{T}\mathbb{Z}_{+}}M(n)$
	be an admissible $g$-twisted $V$-module, the\emph{ contragredient
		module }$M'$ is defined as follows:
	\[
	M'=\bigoplus_{n\in\frac{1}{T}\mathbb{Z}_{+}}M(n)^{*},
	\]
	where $M(n)^{*}=\mbox{Hom}_{\mathbb{C}}(M(n),\mathbb{C}).$ The vertex
	operator $Y_{M'}(v,z)$ is defined for $v\in V$ via
	\begin{eqnarray*}
		\langle Y_{M'}(v,z)f,u\rangle=\langle f,Y_{M}(e^{zL(1)}(-z^{-2})^{L(0)}v,z^{-1})u\rangle,
	\end{eqnarray*}
	where $\langle f,w\rangle=f(w)$ is the natural paring $M'\times M\to\mathbb{C}.$
\end{defn}
\begin{rmk} 1.  $(M',Y_{M'})$ is an admissible $g^{-1}$-twisted
	$V$-module, and $M$ is irreducible if and only if $M'$ is irreducible. \cite{FHL93}.
\end{rmk}

We now recall from \cite{FHL93} the notions of intertwining operators and fusion rules.

\begin{defn} Let $(V,\ Y)$ be a vertex operator algebra and
let $(W^{1},\ Y^{1}),\ (W^{2},\ Y^{2})$ and $(W^{3},\ Y^{3})$ be
$V$-modules. An \emph{intertwining operator} of type $\left(\begin{array}{c}
W^{3}\\
W^{1\ }W^{2}
\end{array}\right)$ is a linear map
\begin{align*}
I(\cdot, z) : W^{1} & \longrightarrow \Hom ( W^{2}, W^{3} ) \{ z \} \\
                 u  & \longmapsto I(u, z) = \sum_{n\in\mathbb{Q}}u_{n}z^{-n-1}
\end{align*}
satisfying:

(1) for any $u\in W^{1}$ and $v\in W^{2}$, $u_{n}v=0$ for $n$
sufficiently large;

(2) $I(L(-1)v,\ z)=\frac{d}{dz}I(v,\ z)$;

(3) (Jacobi identity) for any $u\in V,\ v\in W^{1}$

\[
z_{0}^{-1}\delta\left(\frac{z_{1}-z_{2}}{z_{0}}\right)Y^{3}(u,\ z_{1})I(v,\ z_{2})-z_{0}^{-1}\delta\left(\frac{-z_{2}+z_{1}}{z_{0}}\right)I(v,\ z_{2})Y^{2}(u,\ z_{1})
\]
\[
=z_{2}^{-1}\delta\left(\frac{z_{1}-z_{0}}{z_{2}}\right)I(Y^{1}(u,\ z_{0})v,\ z_{2}).
\]

The space of all intertwining operators of type $\left(\begin{array}{c}
W^{3}\\
W^{1}\ W^{2}
\end{array}\right)$ is denoted by
$$I_{V}\left(\begin{array}{c}
W^{3}\\
W^{1}\ W^{2}
\end{array}\right).$$ Let $N_{W^{1}, W^{2}}^{W^{3}}=\dim I_{V}\left(\begin{array}{c}
W^{3}\\
W^{1}\ W^{2}
\end{array}\right)$. These integers $N_{W^{1}, W^{2}}^{W^{3}}$ are usually called the
\emph{fusion rules}.
\end{defn}




\begin{defn} Let $V$ be a vertex operator algebra, and $W^{1},$
$W^{2}$ be two $V$-modules. A module $(W,I)$, where $I\in I_{V}\left(\begin{array}{c}
\ \ W\ \\
W^{1}\ \ W^{2}
\end{array}\right),$ is called a \emph{tensor product} (or fusion product) of $W^{1}$
and $W^{2}$ if for any $V$-module $M$ and $\mathcal{Y}\in I_{V}\left(\begin{array}{c}
\ \ M\ \\
W^{1}\ \ W^{2}
\end{array}\right),$ there is a unique $V$-module homomorphism $f:W\rightarrow M,$ such
that $\mathcal{Y}=f\circ I.$ As usual, we denote $(W,I)$ by $W^{1}\boxtimes_{V}W^{2}$ or $W^{1}\boxtimes W^{2}$ simply.
\end{defn}


Fusion rules have the following symmetric property \cite{FHL93}.

\begin{prop}\label{fusionsymm.}
Let $W^{i} (i=1,2,3)$ be $V$-modules. Then
$$N_{W^{1},W^{2}}^{W^{3}}=N_{W^{2},W^{1}}^{W^{3}}, \ N_{W^{1},W^{2}}^{W^{3}}=N_{W^{1},(W^{3})^{'}}^{(W^{2})^{'}}.$$
\end{prop}

We next recall from \cite{DLM96-1} intertwining operators among weak ${g}_i$-tiwisted modules $(M_i,Y_{M_i})$ for $i=1,2,3$, where $g_i$, $i=1,2,3$ are commuting automorphisms of $V$ of order $T_i$, $i=1,2,3$. $V$ can be decomposed into the direct sum of commen eigenspaces 
$$
V=\bigoplus_{j_1,j_2}V_{j_1,j_2}, 
$$
where $V_{j_1,j_2}=\{v\in V|g_s\cdot v=e^{\frac{2\pi j_s\sqrt{-1}}{T_s}}v, \ s=1,2\}$. 

\vskip 0.2cm
An intertwining operator of type $\left(\begin{array}{c}
M^{3}\\
M^{1\ }M^{2}
\end{array}\right)$ with the given data is a linear map
\begin{align*}
{\mathcal Y} : \ M^{1} & \longrightarrow \Hom ( M^{2}, M^{3} ) \{ z \} \\
u  & \longmapsto {\mathcal Y}(u, z) = \sum_{n\in\mathbb{C}}u_{n}z^{-n-1}
\end{align*}
such that for $w^i\in M_i$, $i=1,2$, fixed $c\in\C$, and $n\in{\mathbb Q}$ sufficently large
$$
w^1_{n+c}w^2=0,
$$
 the following (generalized) Jacobi identity holds on $M_2$: for $u\in V_{j_1,j_2}$, and $w\in M_1$,
\[
z_{0}^{-1}\delta\left(\frac{z_{1}-z_{2}}{z_{0}}\right)^{j_1/T_1}Y_{M_3}(u,\ z_{1}){\mathcal Y}(v,\ z_{2})-z_{0}^{-1}\delta\left(\frac{-z_{2}+z_{1}}{z_{0}}\right)^{j_1/T_1}{\mathcal Y}(v,\ z_{2})Y_{M_2}(u,\ z_{1})
\]
\[
=z_{2}^{-1}\delta\left(\frac{z_{1}-z_{0}}{z_{2}}\right)^{-j_2/T_2}{\mathcal Y}(Y_{M_1}(u,\ z_{0})v,\ z_{2}),
\]
and 
$$
{\mathcal Y}(L(-1)v,z)=\frac{d}{dz}{\mathcal Y}(v,z).
$$
The dimension $N_{M_1,M_2}^{M_3}$ of the vector space of intertwining operators of type $\left(\begin{array}{c}
M^{3}\\
M^{1\ }M^{2}
\end{array}\right)$  is called the fusion rule. If $N_{M_1,M_2}^{M_3}\geq 1$, then $g_3=g_1g_2$ \cite{DLM96-1}.
\vskip 0.2cm
We finally recall some notions about quantum dimensions \cite{DJX13}.

\begin{defn}Let $M=\oplus_{n\in\frac{1}{T}\mathbb{Z}_{+}}M_{\lambda+n}$
	be a $g$-twisted $V$-module, the \emph{formal character} of $M$
	is defined as
	
	\[
	\mbox{ch}_{q}M=\mbox{tr}_{M}q^{L\left(0\right)-c/24}=q^{\lambda-c/24}\sum_{n\in\frac{1}{T}\mathbb{Z}_{+}}\left(\dim M_{\lambda+n}\right)q^{n},
	\]
	where $c$ is the central charge of the vertex operator algebra $V$
	and $\lambda$ is the conformal weight of $M$. \end{defn}

It is proved in \cite{Z96}, \cite{DLM00} that $\mbox{ch}_{q}M$ converges to a
holomorphic function in the domain $|q|<1.$ We denote the holomorphic
function $\mbox{ch}_{q}M$ by $Z_{M}\left(\tau\right)$. Here and
below, $\tau$ is in the upper half plane $\mathbb{H}$ and $q=e^{2\pi i\tau}$.

Let $V$ be a vertex
	operator algebra and $M$ a $g$-twisted $V$-module such that $Z_{V}\left(\tau\right)$
	and $Z_{M}\left(\tau\right)$ exists. The quantum dimension of $M$
	over $V$ is defined as \cite{DJX13}
	\[
	\mbox{qdim}_{V}M=\lim_{y\to0}\frac{Z_{M}\left(iy\right)}{Z_{V}\left(iy\right)},
	\]
	where $y$ is real and positive.  Let $M^{0}\cong V,\, M^{1},\,\cdots,\, M^{d}$
denote all inequivalent irreducible $V$-modules. Moreover, assume
the conformal weights $\lambda_{i}$ of $M^{i}$ are positive for
all $i>0.$  Then \cite{DJX13}
\[
\mbox{qdim}_{V}\left(M^{i}\boxtimes M^{j}\right)=\mbox{qdim}_{V}M^{i}\cdot \mbox{qdim}_{V}M^{j}
\]
for $i,\, j=0,\cdots,\, d.$

\section{Lemmas and Propositions}

In this section, we determine the  fusion rules of the orbifold vertex opteator algbera $L_{\widehat{\mathfrak{sl}_2}}(k,0)^{K}$. For convenience we simply denote $L_{\widehat{\frak{sl}_{2}}}(k,i)$ by $L(k,i)$, and $L_{\widehat{\frak{sl}_{2}}}(k,0)^{K}$ by $L(k,0)^K$. 

\vskip 0.2cm
We first  recall $L(k,0)^{K}$ and their irreducible modules from \cite{JWa}. 
Let $\{e,f,h\}$ be a stand basis of $\frak{sl}_2(\C)$. Let  $\sigma_1$ and $\sigma_2$ be two automorphismms of $\frak{sl}_2(\C)$ defined by:
$$
\sigma_1(h)=h, \ \sigma_1(e)=-e, \ \sigma_1(f)=-f;
$$
$$
\sigma_2(h)=-h, \ \sigma_2(e)=f, \ \sigma_2(f)=e.
$$
Then 
$\sigma_1$ and $\sigma_2$ can be lifted to  automorphisms of the simple vertex operator algebra $L(k,0)$.  Let $K$ be the automorphism subgroup of $L(k,0)$ generated by $\sigma_1$ and $\sigma_2$. Then $K$ is the Klein group of order 4, since $\sigma_1^2=1, \sigma_2^2=1$, and $\sigma_1\sigma_2=\sigma_2\sigma_1$.  
Denote $$\sigma_3=\sigma_1\sigma_2.$$
Let $\alpha$ be the simple root of $\mathfrak{sl}_2(\C)$ with $(\alpha|\alpha)= 2$. From \cite{FZ92}, the integrable highest weight $L(k,0)$-modules $L(k, i)$ for $0 \leqslant i \leqslant k$ provide a complete list of irreducible $L(k,0)$-modules with the lowest weight spaces being  $(i+1)$-dimensional irreducible $\mathfrak{sl}_2(\C)$-modules $L(\frac{i\alpha}{2})$, respectively.
Set
\[h^{(1)} = h, \quad e^{(1)} = e, \quad f^{(1)} = f,\]
\[h^{(2)} = e + f, \quad e^{(2)} = \frac{1}{2}(h - e + f), \quad f^{(2)} = \frac{1}{2}(h + e - f),\]
\[h^{(3)} = \sqrt{-1}(e - f), \quad e^{(3)} = \frac{1}{2}(\sqrt{-1}h + e + f), \quad f^{(3)} = \frac{1}{2}(-\sqrt{-1}h + e + f).\]
Then $\{ h^{(r)}$, $e^{(r)}$, $f^{(r)} \}$, $r=1,2,3$ are $\mathfrak{sl}_2$-triples, and 
\[\sigma_r(h^{(r)}) = h^{(r)}, \quad \sigma_r(e^{(r)}) = -e^{(r)}, \quad \sigma_r(f^{(r)}) = -f^{(r)}.\]
For $0\leq i\leq k$ and $\sigma_r$, $r=1,2,3$,  let  $v^{r,i,0}$ be a highest weight vector of $\mathfrak{sl}_2(\mathbb{C})$ with respect to the $\mathfrak{sl}_2$-triple $\{ h^{(r)}, e^{(r)}, f^{(r)} \}$ with the weight $i$, and set
\[v^{r,i,j} = \frac{1}{j!}f^{(r)}(0)v^{r,i,0},  \ \ \ 0 \leqslant j \leqslant i.\]
Then 
\[ h^{(r)}(0)v^{r,i,j} = (i - 2j) v^{r,i,j} \quad \text{for} \quad 0 \leqslant j \leqslant i , \]
\[ e^{(r)}(0)v^{r,i,0} = 0, \quad e^{(r)}(0)v^{r,i,j} = (i - j + 1)v^{r,i,j-1} \quad \text{for} \quad 1 \leqslant j \leqslant i , \]
\[ f^{(r)}(0)v^{r,i,i} = 0, \quad f^{(r)}(0)v^{r,i,j} = (j + 1)v^{r,i,j+1} \quad \text{for} \quad 0 \leqslant j \leqslant i-1 , \]
\[ a^{(r)}(n)v^{r,i,j} = 0 \quad \text{for} \quad a \in \{ h, e, f \}, \quad n \geqslant 1 .\]
We denote
$v^{1,i,j}$ by $v^{i,j}$ for $0\leq j\leq i$. Then we may assume that
\[  v^{2,i,i} = \sum_{j=0}^{i}(-1)^jv^{i,j}, \quad v^{3,i,i} = \sum_{j=0}^{i}(\sqrt{-1})^jv^{i,j}. \]
Recall from \cite{JWa}, for $i\in 2\Z_{\geq 2}$, $r=1,2,3$, $L(k,i)^{\sigma_r, j}$, $j=1,2,3,4$ are the irreducible $L(k,0)^{K}$-modules generated by the conformal weight vectors $v^{r,i,0} + v^{r,i,i}$,  $v^{r,i,0} - v^{r,i,i}$, $v^{r,i,1} + v^{r,i,i-1}$, and $v^{r,i,1} - v^{r, i,i-1}$, respectively.  For $i=2$, $L(k,2)^{\sigma_r, j}$, $j=1,2,3,4$ are the irreducible $L(k,0)^{K}$-modules generated by the conformal weight vectors $v^{r,2,0} + v^{r,2,2}$,  $v^{r,2,0} - v^{r,2,2}$, $v^{r,2,1}$, and $h^{(r)}(-1)v^{r,2,1}$, respectively.
 For $i=0$, $L(k,0)^{\sigma_r, j}$, $j=1,2,3,4$ are the irreducible $L(k,0)^{K}$-modules generated by the conformal weight vectors ${\bf 1}$,  $h^{(r)}(-1){\bf 1}$, $(e^{(r)}+f^{(r)})(-1){\bf 1}$, and $(e^{(r)}-f^{(r)})(-1){\bf 1}$, respectively. 
 
 \vskip 0.3cm
 For simplicity, we denote  $L(k,i)^{\sigma_1,l}$ by $L(k,i)^l$, for 
 $0\leq i\leq k$, $i\in 2\Z$, $1\leq l\leq 4$, respectively.  In particular, $L(k,0)^1=L(k,0)^K$, and $L(k,0)^j$, $j=2,3,4$ are generated by  $h(-1){\bf 1}$, $(e(-1)+f(-1)){\bf 1}$, $(e(-1)-f(-1)){\bf 1}$ respectively.
 We have the following result from \cite{JWa}.  
\begin{thm}\label{main1}
	(1) If $k \in 2\mathbb{Z}_{\geqslant 0}+1$, there are $\frac{11(k+1)}{2}$ inequivalent irreducible $L(k,0)^{K}$-modules as follows:
	$$
	L(k, i)^{+}, \ L(k,j)^{l}, \ 1\leqslant i\leqslant k, \ i\in 2\Z+1, \ 0\leqslant j\leqslant k, \ j\in 2\Z, \ 1\leqslant l\leqslant 4,
	$$
	$$
	\overline{L(k,i)}^{\sigma_r,+}, \ \overline{L(k,i)}^{\sigma_r,-}, \ 0\leqslant i\leqslant \frac{k-1}{2}, \ r=1,2,3.
	$$
	(2) If $k \in 2\mathbb{Z}_+$, there are $\frac{11k+32}{2}$ inequivalent irreducible $L(k,0)^{K}$-modules as follows:
	$$
	L(k, i)^{+}, \ L(k,j)^{l}, \ 1\leqslant i\leqslant k, \ i\in 2\Z+1, \ 0\leqslant j\leqslant k, \ j\in 2\Z, \ 1\leqslant l\leqslant 4,
	$$
	$$
	\overline{L(k,i)}^{\sigma_r,+}, \ \overline{L(k,i)}^{\sigma_r,-}, \ \overline{L(k,\frac{k}{2})}^{{\sigma_r}, l}, \ 0\leqslant i\leqslant \frac{k}{2}-1, \ r=1,2,3, \ 1\leqslant l\leqslant 4, 
	$$
	where $\overline{L(k,\frac{k}{2})}^{{\sigma_r,1}}+\overline{L(k,\frac{k}{2})}^{{\sigma_r,2}}=\overline{L(k,\frac{k}{2})}^{{\sigma_r},+}$,  \  $\overline{L(k,\frac{k}{2})}^{{\sigma_r,3}}+\overline{L(k,\frac{k}{2})}^{{\sigma_r,4}}=\overline{L(k,\frac{k}{2})}^{{\sigma_r},-}$.
	\end{thm}
It is easy to see that for $r=1,2,3$, 
\begin{equation}\label{e1}
L(k,0)^{\sigma_r,j}\boxtimes \overline{L(k,i)}^{\sigma_r,\pm}= \overline{L(k,i)}^{\sigma_r,\pm}, \ j=1,2, 
\end{equation}
\begin{equation}\label{e2}
L(k,0)^{\sigma_r,j}\boxtimes \overline{L(k,i)}^{\sigma_r,\pm}= \overline{L(k,i)}^{\sigma_r,\mp}, \ j=3,4, 
\end{equation}
\begin{equation}\label{e3}
L(k,0)^{\sigma_r,j}\boxtimes \overline{L(k,\frac{k}{2})}^{\sigma_r,1}= \overline{L(k,\frac{k}{2})}^{\sigma_r,j}, \ j=1,2,3,4.
\end{equation}
We first  have the following lemmas.
\begin{lem}\label{l3.7} For $i\in 2\Z_{\geq 0}$,  we have
	$$
	L(k,0)^r\boxtimes L(k,i)^1=L(k,i)^r, \  r=1,2,3,4,
	$$
	$$
	L(k,0)^r\boxtimes L(k,i)^r=L(k,i)^1, \  r=1,2,3,4,
	$$
	$$
	L(k,0)^r\boxtimes L(k,i)^s=L(k,i)^t, \  \{r,s,t\}=\{2,3,4\}.
	$$
\end{lem}
\begin{proof}The first  formula follows from the definition of $L(k,i)^r$, for $0\leq i\leq k$, $i\in2\Z$, $0\leq r\leq 4$.  For other relations, notice that   $L(k,0)^2$ is generated by $h(-1){\bf 1}$, and $L(k,i)^2$ is generated by $v^{i,j}-v^{i,i-j}$,  where $j\in 2\Z$, $0\leq j\leq i$. Since
	$$
	h(0)(v^{i,j}-v^{i,i-j})=(i-2j)(v^{i,j}+v^{i,i-j}),
	$$
it follows that 	$
L(k,0)^2\boxtimes L(k,i)^2=L(k,i)^1$.  By  definition, $L(k,0)^3$ is generated by $(e(-1)+f(-1)){\bf 1}$, and  $L(k,i)^3$ is generated by $v^{i,j}-v^{i,i-j}$, where  $j\in 2\Z+1$,  $0\leq j\leq i$. Notice that
	$$
(e+f)(0)(v^{i,j}+v^{i,i-j})=(i-j+1)(v^{i,j-1}+v^{i,i-j+1})+(j+1)(v^{i,j+1}+v^{i,i-j-1})\in L(k,i)^1.
$$
This means that $L(k,0)^3\boxtimes L(k,i)^3=L(k,i)^1$.  Obviously, the formula that $L(k,0)^4\boxtimes L(k,i)^4=L(k,i)^1$ follows from the fact that
	$$
(e-f)(0)(v^{i,j}-v^{i,i-j})=(i-j+1)(v^{i,j-1}+v^{i,i-j+1})-(j+1)(v^{i,j+1}+v^{i,i-j-1})\in L(k,i)^1.
$$
The proof for the last  relation is similar.
	\end{proof}
\begin{lem}\label{l3.8}
	For $i\in 4\Z+2$, we have 
	$$
	L(k,i)^1=L(k,i)^{\sigma_2,1}=L(k,i)^{\sigma_3,3}, 
	$$
	$$
	L(k,i)^2=L(k,i)^{\sigma_2,3}=L(k,i)^{\sigma_3,2}, 
	$$
	$$
	L(k,i)^3=L(k,i)^{\sigma_2,2}=L(k,i)^{\sigma_3,1}, 
	$$
	$$
	L(k,i)^4=L(k,i)^{\sigma_2,4}=L(k,i)^{\sigma_3,4}.
	$$
	For $i\in 4\Z$, we have
	$$
	L(k,i)^1=L(k,i)^{\sigma_2,1}=L(k,i)^{\sigma_3,1}, 
	$$
	$$
	L(k,i)^2=L(k,i)^{\sigma_2,3}=L(k,i)^{\sigma_3,4}, 
	$$
	$$
	L(k,i)^3=L(k,i)^{\sigma_2,2}=L(k,i)^{\sigma_3,3}, 
	$$
	$$
	L(k,i)^4=L(k,i)^{\sigma_2,4}=L(k,i)^{\sigma_3,2}.
	$$	
\end{lem}
\begin{proof}
	Notice that $ v^{2,i,i} = \sum_{j=0}^{i}(-1)^jv^{i,j}, \quad v^{3,i,i} = \sum_{j=0}^{i}(\sqrt{-1})^jv^{i,j}.$
	Direct calculation yields that  for $i\in 2\Z$, 
	$$
	v^{2,i,0} = \sum_{j=0}^{i}v^{i,j},  \quad   v^{3,i,0} = \sum_{j=0}^{i}(-\sqrt{-1})^{i+j}v^{i,j}.
	$$
	Then we have for $i\in 2\Z$, 
	$$
	v^{2,i,i}+v^{2,i,0}=\sum\limits_{\tiny{\begin{split}0\leq j\leq i \\ j\in 2\Z\end{split}}}2v^{i,j}, \quad v^{2,i,0}-v^{2,i,i}=\sum_{\tiny{\begin{split}0\leq j\leq i \\ j\in 2\Z+1\end{split}}}2v^{i,j}$$
	and if $i\in4\Z+2$, 
	$$ v^{3,i,i}+v^{3,i,0}=\sum_{\tiny{\begin{split}0\leq j\leq i\\ j\in 2\Z+1\end{split}}}2(\sqrt{-1})^jv^{i,j}=\sum_{\tiny{\begin{split}0\leq j< \frac{i}{2}\\ j\in 2\Z+1\end{split}}}2(\sqrt{-1})^j(v^{i,j}+v^{i,i-j})+2(\sqrt{-1})^{\frac{i}{2}}v^{i,\frac{i}{2}},$$
	$$
	v^{3,i,0}-v^{3,i,i}=\sum_{\tiny{\begin{split}0\leq j\leq i\\j\in 2\Z\end{split}}}2(\sqrt{-1})^jv^{i,j}
	=\sum_{\tiny{\begin{split}0\leq j< \frac{i}{2}\\ j\in 2\Z\end{split}}}2(\sqrt{-1})^j(v^{i,j}-v^{i,i-j}).
	$$	
	If $i\in4\Z$, then 
	$$ v^{3,i,i}+v^{3,i,0}=\sum\limits_{\tiny{\begin{split}0\leq j\leq i \\ j\in 2\Z\end{split}}}2(\sqrt{-1})^jv^{i,j}=\sum_{\tiny{\begin{split}0\leq j< \frac{i}{2}\\ j\in 2\Z\end{split}}}2(\sqrt{-1})^j(v^{i,j}+v^{i-j})+2(\sqrt{-1})^{\frac{i}{2}}v^{i,\frac{i}{2}}, $$$$ 
	v^{3,i,0}-v^{3,i,i}=\sum_{\tiny{\begin{split}0\leq j\leq i\\ j\in 2\Z+1\end{split}}}2(\sqrt{-1})^jv^{i,j}=\sum_{\tiny{\begin{split}0\leq j< \frac{i}{2}\\ j\in 2\Z+1\end{split}}}2(\sqrt{-1})^j(v^{i,j}-v^{i,i-j}).
	$$
	This means that for $i\in 2\Z$,
	$$L(k,i)^{\sigma_2,j}=L(k,i)^{j},  \ 
	L(k,i)^{\sigma_2,2}=L(k,i)^{3},  \  L(k,i)^{\sigma_2,3}=L(k,i)^2, \ j=1,4, 
	$$
	and
	$$
	L(k,i)^{\sigma_3,j}=L(k,i)^{j},  \ 
	L(k,i)^{\sigma_3,1}=L(k,i)^{3},  \ L(k,i)^{\sigma_3,3}=L(k,i)^{1}, \ j=2,4 \ {\rm if} \ i\in 4\Z+2,
	$$
	$$
	L(k,i)^{\sigma_3,j}=L(k,i)^{j},  \ 
	L(k,i)^{\sigma_3,2}=L(k,i)^{4},  \ 	L(k,i)^{\sigma_3,4}=L(k,i)^{2}, \ j=1,3, \  {\rm if} \ i\in 4\Z.
	$$
	The other relations follow from Lemma \ref{lem3.7}.
\end{proof}
We now recall some results from \cite{DLY}, \cite{DLWY}, \cite{ALY1}, \cite{JW1},  and \cite{JW2}.  Let $K_0=K(\frak{sl_2},k)$ be the commutant vertex oparator algbera of the Heisenberg vertex  subalgebra  of $L(k,0)$ generated by $h(-1){\bf 1}$. $K_0$ is called the parafermion vertex operator algbera associated to $L(k,0)$. The irreducible $K_0$-modules $M^{i,j}$, for $0\leq i\leq k$, $0\leq j\leq k-1$ were constructed in \cite{DLY}, where  $K_0=M^{0,0}$, and $M^{i,j}\cong M^{k-i,k-i+j}$. Theorem 8.2 in \cite{ALY1} showed that the $\frac{k(k+1)}{2}$ irreducible $K_0$-modules $M^{i,j}$ for $1\leq i\leq k, 0\leq j\leq i-1$ constructed in \cite{DLY} form a complete set of isomorphism classes of irreducible $K_0$-modules. Moreover, $K_0$ is $C_2$-cofinite \cite{ALY1} and rational \cite{ALY2}.

Recall from \cite{DLY}, for $0\leq i\leq k$, 
\begin{eqnarray}
L(k,i)=\bigoplus_{j=0}^{k-1}V_{\mathbb{Z}\gamma+(i-2j)\gamma/2k}\otimes M^{i,j} \ \ \label{eq:3.1}
\end{eqnarray}
where $\gamma=h(-1){\bf 1}$ so that $(\gamma|\gamma)=2k$. 
Let $K_0^+=K_0^{\sigma_2}$ be the orbifold parafermion vertex operator algebra in \cite{JW1} and \cite{JW2}. Then $K_0=K_0^+\oplus K_0^-$. We have the following result from \cite{JW1} and \cite{JW2}.
 \begin{thm}\label{thm:orbifold3'}
	If $k=2n+1$, $n\geq 1$, there are $\frac{(k+1)(k+7)}{4}$ inequivalent irreducible modules of $K_{0}^+$.
	If $k=2n$, $n\geq 2$, there are $\frac{(k^{2}+8k+28)}{4}$ inequivalent irreducible $K_0^+$-modules. More precisely, if $k=2n+1$, $n\geq 1$, the set
	\begin{eqnarray*}
		&& \{ W(k,i)^{j} \ \mbox{for} \  0\leq i\leq \frac{k-1}{2}, j=1,2,\\
		&&(M^{i,j})^{s} \ \mbox{for} \  (i,j)=(i,\frac{i}{2}), i=2,4,6,\cdots,2n, \ \mbox{and} \ (i,j)=(2n+1,0), s=0,1,\\ && M^{i,0} \ \mbox{for} \ 1\leq i\leq \frac{k-1}{2},
		M^{i,j} \ \mbox{for} \ 3\leq i\leq k, \mbox{if} \  i=2m, 1\leq j\leq m-1, \mbox{if} \ i=2m+1, 1\leq j\leq m\} \\
	\end{eqnarray*} gives all inequivalent irreducible $K_{0}^+$-modules. If $k=2n$, $n\geq 2$, 
	\begin{eqnarray*}
		&&  W(k,i)^{j} \ \mbox{for} \  0\leq i\leq \frac{k}{2}, j=1,2,  \widetilde{W(k,\frac{k}{2})}^{j} \ \mbox{for} \ j=1,2,\\
		&&(M^{i,j})^{s} \ \mbox{for} \  (i,j)=(i,\frac{i}{2}), i=2,4,6,\cdots,2n, (i,j)=(n,0) \mbox{and} \ (i,j)=(2n,0), s=0,1,\\ && M^{i,0} \ \mbox{for} \ 1\leq i\leq \frac{k-2}{2},
		M^{i,j} \ \mbox{for} \ 3\leq i\leq k, \mbox{if} \  i=2m, 1\leq j\leq m-1, \mbox{if} \ i=2m+1, 1\leq j\leq m\\
	\end{eqnarray*} exhaust  all inequivalent irreducible $K_{0}^+$-modules.
\end{thm}
Recall from \cite{DLM00} and \cite{DLWY}, $V_{\Z\gamma}$ is generated by $e^{\gamma}=\frac{1}{k!}e(-1)^k{\bf 1}$ and $e^{-\gamma}=\frac{1}{k!}f(-1)^k{\bf 1}$, and $K_0$ is generated by \begin{equation}\label{W1}
\begin{split}
W^3 &= k^2 h(-3)\1 + 3 k h(-2)h(-1)\1 +
2h(-1)^3\1 - 6k h(-1)e(-1)f(-1)\1 \\
& \quad + 3 k^2e(-2)f(-1)\1 - 3 k^2e(-1)f(-2)\1.
\end{split}
\end{equation}
Notice that $V_{\Z\gamma}^+$ is generated by $e^{\gamma}+e^{-\gamma}$ \cite{DN99}, and $K_0^+$ is  generated by $W^3_1W^3$ \cite{JW2}. Then it is easy to see that $e^{\gamma}+e^{-\gamma}$, $W^3_1W^3\in L(k,0)^K$, for $k\in 2\Z_+$. It follows that  $V_{\Z\gamma}^+\otimes K_0^+\subseteq L(k,0)^K$. The following lemma is easy to check. 
\begin{lem}\label{l3.14} As $V_{\Z\gamma}^+\otimes K_0^+$-modules, for $k\in 4\Z+2$, we have 
$$
\begin{array}{ll}
 L(k,0)^K=&L(k,0)^1=V_{\Z\gamma}^+\otimes K_0^+\oplus V_{\Z\gamma}^-\otimes K_0^-
\\&\\
&\oplus[\oplus_{\tiny{\begin{split}2\leq i<\frac{k}{2} \\  i\in 2\mathbb{Z} \ \  \end{split}}} (V_{\Z\gamma-\frac{i}{k}\gamma}\otimes M^{0,i}\oplus 
V_{\Z\gamma+\frac{i}{k}\gamma}\otimes M^{0,k-i})^+],\\
 L(k,0)^2=&V_{\Z\gamma}^+\otimes K_0^-\oplus V_{\Z\gamma}^-\otimes K_0^+\oplus 
[\oplus_{\tiny{\begin{split}2\leq i<\frac{k}{2} \\  i\in 2\mathbb{Z} \ \  \end{split}}} (V_{\Z\gamma-\frac{i}{k}\gamma}\otimes M^{0,i}\oplus 
V_{\Z\gamma+\frac{i}{k}\gamma}\otimes M^{0,k-i})^-],
\\
L(k,0)^3=&\oplus_{\tiny{\begin{split}2\leq i<\frac{k}{2} \\  i\in 2\mathbb{Z} +1\ \  \end{split}}} (V_{\Z\gamma-\frac{i}{k}\gamma}\otimes M^{0,i}\oplus 
V_{\Z\gamma+\frac{i}{k}\gamma}\otimes M^{0,k-i})^+
\\
&  \oplus V_{\Z\gamma+\frac{1}{2}\gamma}^+\otimes  (M^{0,\frac{k}{2}})^+
\oplus V_{\Z\gamma+\frac{1}{2}\gamma}^-\otimes  (M^{0,\frac{k}{2}})^-,\\
&\\
L(k,0)^4=&\oplus_{\tiny{\begin{split}2\leq i<\frac{k}{2} \\  i\in 2\mathbb{Z} +1\ \  \end{split}}} (V_{\Z\gamma-\frac{i}{k}\gamma}\otimes M^{0,i}\oplus 
V_{\Z\gamma+\frac{i}{k}\gamma}\otimes M^{0,k-i})^-\\
& \oplus V_{\Z\gamma+\frac{1}{2}\gamma}^+\otimes  (M^{0,\frac{k}{2}})^-
\oplus V_{\Z\gamma+\frac{1}{2}\gamma}^-\otimes  (M^{0,\frac{k}{2}})^+;
\end{array}
$$
and for $k\in 4\Z$, 
$$
\begin{array}{ll}
L(k,0)^K= & V_{\Z\gamma}^+\otimes K_0^+\oplus V_{\Z\gamma}^-\otimes K_0^-\oplus 
[\oplus_{\tiny{\begin{split}2\leq i<\frac{k}{2} \\  i\in 2\mathbb{Z} \ \  \end{split}}}(V_{\Z\gamma-\frac{i}{k}\gamma}\otimes M^{0,i}\oplus 
V_{\Z\gamma+\frac{i}{k}\gamma}\otimes M^{0,k-i})^+]\\
& \oplus V_{\Z\gamma+\frac{1}{2}\gamma}^+\otimes  (M^{0,\frac{k}{2}})^+
\oplus V_{\Z\gamma+\frac{1}{2}\gamma}^-\otimes  (M^{0,\frac{k}{2}})^-,
\end{array}
$$
$$
\begin{array}{ll}
L(k,0)^2= & V_{\Z\gamma}^+\otimes K_0^-\oplus V_{\Z\gamma}^-\otimes K_0^+\oplus 
[\oplus_{\tiny{\begin{split}2\leq i<\frac{k}{2} \\  i\in 2\mathbb{Z} \ \  \end{split}}}(V_{\Z\gamma-\frac{i}{k}\gamma}\otimes M^{0,i}\oplus 
V_{\Z\gamma+\frac{i}{k}\gamma}\otimes M^{0,k-i})^-]\\
& \oplus V_{\Z\gamma+\frac{1}{2}\gamma}^+\otimes  (M^{0,\frac{k}{2}})^-
\oplus V_{\Z\gamma+\frac{1}{2}\gamma}^-\otimes  (M^{0,\frac{k}{2}})^+,
\\
&\\
L(k,0)^3  =& \oplus_{\tiny{\begin{split}1\leq i<\frac{k}{2} \\  i\in 2\mathbb{Z} +1\ \  \end{split}}} (V_{\Z\gamma-\frac{i}{k}\gamma}\otimes M^{0,i}\oplus 
V_{\Z\gamma+\frac{i}{k}\gamma}\otimes M^{0,k-i})^+,
\\
&\\
L(k,0)^4 =&\oplus_{\tiny{\begin{split}1\leq i<\frac{k}{2} \\  i\in 2\mathbb{Z} +1\ \  \end{split}}} (V_{\Z\gamma-\frac{i}{k}\gamma}\otimes M^{0,i}\oplus 
V_{\Z\gamma+\frac{i}{k}\gamma}\otimes M^{0,k-i})^-,
\end{array}
$$
where  $(V_{\Z\gamma-\frac{i}{k}\gamma}\otimes M^{0,i}\oplus 
V_{\Z\gamma+\frac{i}{k}\gamma}\otimes M^{0,k-i})^+$  is generated by $v^{k,i}+v^{k,k-i}$ as an $V_{\Z\gamma}^+\otimes K_0^+$-module, and $(V_{\Z\gamma-\frac{i}{k}\gamma}\otimes M^{0,i}\oplus 
V_{\Z\gamma+\frac{i}{k}\gamma}\otimes M^{0,k-i})^-$ is generated by $v^{k,i}-v^{k,k-i}$.
\end{lem}
Let $V_{\Z\gamma}^{T_i}$, $i=1,2$  be the $\sigma_2$-twisted $V_{\Z\gamma}$-modules \cite{D}. From \cite{JW2},  for $ 0\leq i\leq k$, $i\neq \frac{k}{2}$, we have  
\begin{eqnarray}\label{JWequation-1}
\overline{L(k,i)}^{\sigma_2}=V_{\mathbb{Z}\gamma}^{T_{a_{i}}}\otimes W(k,i),\label{eq:3.2}
\end{eqnarray}
where $a_i=1$ or $2$, and  if $k\in 2\Z$, 
\begin{eqnarray}
\overline{L(k,\frac{k}{2})}^{\sigma_2}=V_{\mathbb{Z}\gamma}^{T_{a_{\frac{k}{2}}}}\otimes W(k,\frac{k}{2})+V_{\mathbb{Z}\gamma}^{T^{'}_{a_{\frac{k}{2}}}}\otimes \widetilde{W(k,\frac{k}{2})},\label{eq:4.10.}
\end{eqnarray}
where $V_{\mathbb{Z}\gamma}^{T_{a_{\frac{k}{2}}}}, V_{\mathbb{Z}\gamma}^{T^{'}_{a_{\frac{k}{2}}}}\in \{V_{\mathbb{Z}\gamma}^{T_{1}},\ V_{\mathbb{Z}\gamma}^{T_{2}}\}$.
We have the following lemmas.
\begin{lem}\label{lem01}
	 For $k\in\Z_+$ and $0\leq i\leq k$, we have 
\begin{eqnarray}
\overline{L(k,i)}^{\sigma_2}=V_{\mathbb{Z}\gamma}^{T_1}\otimes W(k,i), \  i\in 2\Z, \ i\neq \frac{k}{2},
\label{eq:3.10}
\end{eqnarray}
\begin{eqnarray}
\overline{L(k,i)}^{\sigma_2}=V_{\mathbb{Z}\gamma}^{T_2}\otimes W(k,i), \  i\in 2\Z+1, \   i\neq \frac{k}{2},
\label{eq:3.6}
\end{eqnarray}
\begin{eqnarray}
\overline{L(k,\frac{k}{2})}^{\sigma_2}=V_{\mathbb{Z}\gamma}^{T_2}\otimes (W(k,\frac{k}{2})\oplus \widetilde{W(k,\frac{k}{2})}), \  k\in 4\Z+2,
\label{eq:3.13}
\end{eqnarray}
and
\begin{eqnarray}
\overline{L(k,\frac{k}{2})}^{\sigma_2}=V_{\mathbb{Z}\gamma}^{T_1}\otimes (W(k,\frac{k}{2})\oplus \widetilde{W(k,\frac{k}{2})}), \quad k\in 4\Z,
\label{eq:3.12}
\end{eqnarray}	
as $V_{\Z\gamma}^+\otimes K_0^+$-modules.
\end{lem}
\begin{proof}
We first prove (\ref{eq:3.10}) for $i=0$. It was shown in \cite{JW2} that $$\overline{L(k,0)}^{\sigma_2}=V_{\mathbb{Z}\gamma}^{T_s}\otimes W(k,0),$$
where $s=1$ or $2$. We now prove that $s=1$. We need to consider the twisted action of $(e^{\gamma}+e^{-\gamma})_{k-1}=\frac{1}{k!}(e(-1)^k{\bf 1}+f(-1)^k{\bf 1})_{k-1}$ on the lowest weight vector  ${\bf 1}$ of $\overline{L(k,0)}^{\sigma_2}$. Recall  from \cite{JWa}  that  $h^{(2)}=e+f$. Then using the $\Delta(\cdot,z)$-system introduced in \cite{Li97-2}, we have
$$
\begin{array}{ll}
&Y_{\sigma_2}(e(-1)^k{\bf 1}+f(-1)^k{\bf 1}, z)\\&\\
=&
Y(z^{\frac{1}{4}(e+f)(0)}{\rm exp}(\sum\limits_{n=1}^{\infty}\frac{(e+f)(n)}{-4n}(-z)^{-n})(e(-1)^k{\bf 1}+f(-1)^k{\bf 1}),z).
\end{array}
$$
Direct calculation yields that for $n\geq 2$, 
$$
(e(n)+f(n))(e(-1)^k{\bf 1}+f(-1)^k{\bf 1})=0,
$$
and for $0\leq m\leq k$, 
$$
f(1)^{m}e(-1)^k{\bf 1}=\frac{m!k!}{(k-m)!}e(-1)^{k-m}{\bf 1}, \    \ e(1)^{m}f(-1)^k{\bf 1}=\frac{m!k!}{(k-m)!}f(-1)^{k-m}{\bf 1}.
$$
Then we have 
$$ \begin{array}{ll}
&{\rm exp}(\sum\limits_{n=1}^{\infty}\frac{(e+f)(n)}{-4n}(-z)^{-n})(e(-1)^k{\bf 1}+f(-1)^k{\bf 1})\\&\\
=& \sum\limits_{m=0}^{k-1}\frac{k!}{4^m(k-m)!}z^{-m}(e(-1)^{k-m}{\bf 1}+f(-1)^{k-m}{\bf 1})+\frac{k!}{2^{2k-1}}z^{-k}{\bf 1}.
\end{array}$$
For $n\geq 1$, 
$$
\begin{array}{ll}
& z^{\frac{1}{4}(e+f)(0)}(e(-)^n{\bf 1}+f(-1)^n{\bf 1})=z^{\frac{1}{4}(e+f)(0)}[\frac{1}{2^n}(h^{(2)}(-1)-e^{(2)}(-1)+f^{(2)}(-1))^n{\bf 1}\\&\\&
+\frac{1}{2^n}(h^{(2)}(-1)+e^{(2)}(-1)-f^{(2)}(-1))^n{\bf 1}]= \frac{1}{2^n}\sum\limits_{j=0}^{[\frac{n}{2}]}z^ju^{(n,j)}
+\frac{1}{2^n}\sum\limits_{j=0}^{[\frac{n}{2}]}z^{-j}v^{(n,j)}
\end{array}
$$ 
for some $u^{(n,j)}, v^{(n,j)}\in L(k,0)$ of weight $n$. So
$$
\begin{array}{ll}
& \frac{1}{k!}(e(-1)^k{\bf 1}+f(-1)^k{\bf 1})_{k-1}{\bf 1}
= \sum\limits_{m=0}^{k-1}\frac{1}{4^m(k-m)!}[\sum\limits_{j=0}^{[\frac{k-m}{2}]}u^{(m,j)}(k-m+j-1)
\\&\\&+\sum\limits_{j=1}^{[\frac{k-m}{2}]}v^{(m,j)}(k-m-j-1)]{\bf 1}+ \frac{1}{2^{2k-1}}{\bf 1}=\frac{1}{2^{2k-1}}{\bf 1},
\end{array}
$$
since $k-m-j-1\geq 0$.
This together with (\ref{JWequation-1}) proves that for $k\in\Z_+$, 
$$
\overline{L(k,0)}^{\sigma_2}=V_{\mathbb{Z}\gamma}^{T_1}\otimes W(k,0).
$$	
Notice that for $0\leq i\leq k$, 
$$L(k,i)=\bigoplus_{j=0}^{k-1}V_{\mathbb{Z}\gamma+(i-2j)\gamma/2k}\otimes M^{i,j},$$
and $$ L(k,i)\boxtimes \overline{L(k,0)}^{\sigma_2}=\overline{L(k,i)}^{\sigma_2}.
$$
Then the other relations follow from  the fact that \cite{Ab01}
$$
V_{\Z\gamma+\frac{i}{2k}\gamma}\boxtimes V_{\Z\gamma}^{T_1}=V_{\Z\gamma}^{T_1}, \ {\rm if} \ i \ {\rm is \ even, }
$$
 and 
$$
V_{\Z\gamma+\frac{i}{2k}\gamma}\boxtimes V_{\Z\gamma}^{T_1}=V_{\Z\gamma}^{T_2}, \ {\rm if} \ i \ 
{\rm  is \ odd.}$$
	\end{proof}
\begin{lem}\label{l3.1}	Assume that $k\in 2\Z_+$, then as $V_{\Z\gamma}^+\otimes K_0^+$-modules,   we have  for $0\leq i\leq k$, 
	\begin{eqnarray}
	\overline{L(k,i)}^{\sigma_1}=\oplus_{j=0}^{k-1}V_{\mathbb{Z}\gamma+(\frac{k-2i}{4k}-\frac{j}{k})\gamma}\otimes M^{i,j},   \quad 0\leq i\leq k.
	\label{eq3.7}
	\end{eqnarray}
\end{lem}
\begin{proof} 
 Recall from \cite{JWa} that  $\sigma_1=e^{2\pi ih'(0)}$, where  $h'=\frac{1}{4}h$.  The  $\sigma_1$-twisted $L(k,0)$-module $\overline{L(k,i)}^{\sigma_1}$ can be  constructed through  $\Delta(\cdot,z)$-system  \cite{Li97-2}. Let
\[\Delta(h', z) = z^{h'(0)}\exp(\sum^{\infty}_{n=1}\frac{h'(n)}{-n}(-z)^{-n}).\]
Then $(\overline{L(k,i)}^{\sigma_1}, Y_{\sigma_1}(\cdot,z))=(L(k,i),Y(\Delta(h^{'},z)\cdot,z))$ is an irreducible $\sigma_1$-twisted $L(k,0)$-module \cite{Li97-2}. 
By (\ref{W1}) and the fact that $V_{\Z\gamma}$ is generated by $e^{\gamma}=\frac{1}{k!}e(-1)^k{\bf 1}$ and $e^{-\gamma}=\frac{1}{k!}f(-1)^k{\bf 1}$,  we have for $k\in 2\Z_+$,  $V_{\Z\gamma}\otimes K_0\subseteq L(k,0)^{\sigma_1,+}=L(k,0)^+$. Then  $\overline{L(k,i)}^{\sigma_1}$  can be decomposed into direct sum of irreducible modules of $V_{\Z\gamma}\otimes K_0$. 
As a $\sigma_1$-twisted $L(k,0)$-module,  $\overline{L(k,i)}^{\sigma_1}$ is generated by $v^{i,i}$. Notice that \cite{JWa}
$$
Y_{\sigma_1}(h,z)=Y(h+\frac{k}{2}z^{-1},z),
$$
and for $m\geq 0$, 
$$h'(m)W^3=0.$$
Then it is easy to see that
$$h_0(v^{i,i})=(\frac{k}{2}-i)v^{i,i},$$
and 
$$
Y_{\sigma_1}(W^3, z)=Y(W^3,z).
$$
This implies that $V_{\mathbb{Z}\gamma+(\frac{k-2i}{4k})\gamma}\otimes M^{i,i}\subseteq \overline{L(k,i)}^{\sigma_1}$. Recall from \cite{JW1} that  as $K_0^+$-modules, $M^{i,i}\cong M^{i,0}$. Together with the fact that 
$$	L(k,0)=\bigoplus_{j=0}^{k-1}V_{\mathbb{Z}\gamma+\frac{-j}{k}\gamma}\otimes M^{0,j},$$
we deduce that  as a $V_{\Z\gamma}^+\otimes K_0^+$-module, 
	$$\overline{L(k,i)}^{\sigma_1}=\oplus_{j=0}^{k-1}V_{\mathbb{Z}\gamma+(\frac{k-2i}{4k}-\frac{j}{k})\gamma}\otimes M^{i,j}.$$
\end{proof}
\begin{lem}\label{lem02}
Assume that $k\in 2\Z_+$, then as $V_{\Z\gamma}^+\otimes K_0^+$-modules,   we have  for $0\leq i\leq k$, 	
	\begin{eqnarray}
	\overline{L(k,i)}^{\sigma_3}=V_{\mathbb{Z}\gamma}^{T_2}\otimes W(k,i),  \ k+2i\in 4\Z+2,   
	\label{eq:3.11}
	\end{eqnarray}
	\begin{eqnarray}
	\overline{L(k,i)}^{\sigma_3}=V_{\mathbb{Z}\gamma}^{T_1}\otimes W(k,i),  \ k+2i\in 4\Z,  \   i\neq \frac{k}{2},
	\label{eq:3.7}
	\end{eqnarray}
	\begin{eqnarray}
	\overline{L(k,\frac{k}{2})}^{\sigma_3}=V_{\mathbb{Z}\gamma}^{T_1}\otimes (W(k,\frac{k}{2})\oplus \widetilde{W(k,\frac{k}{2})}).
	\label{eq:3.9}
	\end{eqnarray}
	\end{lem}
\begin{proof} Notice that 
$$
\begin{array}{ll}
&Y_{\sigma_3}(e(-1)^k{\bf 1}+f(-1)^k{\bf 1}, z)\\&\\
=&
Y(z^{\frac{1}{4}\sqrt{-1}(e-f)(0)}{\rm exp}(\sum\limits_{n=1}^{\infty}\frac{\sqrt{-1}(e-f)(n)}{-4n}(-z)^{-n})(e(-1)^k{\bf 1}+f(-1)^k{\bf 1}),z).
\end{array}
$$
Then similar to the proof of Lemma \ref{lem01}, we have 
$$
(e(-1)^k{\bf 1}+f(-1)^k{\bf 1})_{k-1}{\bf 1}=\frac{1}{2^{2k-1}}(\sqrt{-1})^k{\bf 1}.
$$
This deduces  that 
$$
	\overline{L(k,0)}^{\sigma_3}=V_{\mathbb{Z}\gamma}^{T_1}\otimes W(k,0),  \ k\in 4\Z_+, 
$$
$$
\overline{L(k,0)}^{\sigma_3}=V_{\mathbb{Z}\gamma}^{T_2}\otimes W(k,0),  \ k\in 4\Z_++2.
$$
Then similar to the proof in Lemma \ref{lem01}, by considering $L(k,i)\boxtimes \overline{L(k,0)}^{\sigma_3}$ and using the  fusion rules of $V_{\Z\gamma}^+$ \cite{Ab01}, we get  the desired realtions.
\end{proof}
The following two lemmas follow from Lemma \ref{l3.14}-Lemma \ref{lem02}. 
\begin{lem}\label{l3.16} For $k\in 2\Z_+$,  and $0\leq i\leq k$, $i\neq \frac{k}{2}$, as $V_{\Z\gamma}^+\otimes K_0^+$-modules, we have  
$$
\begin{array}{ll}
\overline{L(k,i)}^{\sigma_1,+}=&\sum\limits_{\tiny\begin{split}0\leq j\leq k-1\\ j\in 2\Z \ \ \ \end{split}}V_{\mathbb{Z}\gamma+(\frac{k-2i}{4k}-\frac{j}{k})\gamma}\otimes M^{i,j}, \\&\\
\overline{L(k,i)}^{\sigma_1,-}=&\sum\limits_{\tiny\begin{split}0\leq j\leq k-1\\ j\in 2\Z+1 \ \ \ \end{split}}V_{\mathbb{Z}\gamma+(\frac{k-2i}{4k}-\frac{j}{k})\gamma}\otimes M^{i,j},
\end{array}
$$$$
\begin{array}{ll}
\overline{L(k,i)}^{\sigma_2,+}=&V_{\mathbb{Z}\gamma}^{T_2,+}\otimes W(k,i)^+\oplus V_{\mathbb{Z}\gamma}^{T_2,-}\otimes W(k,i)^-, \ i\in 2\Z+1, \\&\\
\overline{L(k,i)}^{\sigma_2,-}=& V_{\mathbb{Z}\gamma}^{T_2,+}\otimes W(k,i)^-\oplus V_{\mathbb{Z}\gamma}^{T_2,-}\otimes W(k,i)^+, \ i\in 2\Z+1, 
\end{array}
$$$$
\begin{array}{ll}
\overline{L(k,i)}^{\sigma_2,+}=&V_{\mathbb{Z}\gamma}^{T_1,+}\otimes W(k,i)^+\oplus V_{\mathbb{Z}\gamma}^{T_1,-}\otimes W(k,i)^-, \ i\in 2\Z, \\&\\
\overline{L(k,i)}^{\sigma_2,-}=&V_{\mathbb{Z}\gamma}^{T_1,+}\otimes W(k,i)^-\oplus V_{\mathbb{Z}\gamma}^{T_1,-}\otimes W(k,i)^+, \ i\in 2\Z, 
\end{array}
$$$$
\begin{array}{ll}
\overline{L(k,i)}^{\sigma_3,+}=&V_{\mathbb{Z}\gamma}^{T_1,+}\otimes W(k,i)^+\oplus V_{\mathbb{Z}\gamma}^{T_1,-}\otimes W(k,i)^-, \ k+2i\in 4\Z, 
\\&\\
\overline{L(k,i)}^{\sigma_3,-}=&V_{\mathbb{Z}\gamma}^{T_1,+}\otimes W(k,i)^-\oplus V_{\mathbb{Z}\gamma}^{T_1,-}\otimes W(k,i)^+, \ k+2i\in 4\Z, \\
&\\
\overline{L(k,i)}^{\sigma_3,+}=&V_{\mathbb{Z}\gamma}^{T_2,+}\otimes W(k,i)^+\oplus V_{\mathbb{Z}\gamma}^{T_2,-}\otimes W(k,i)^-, \ k+2i\in 4\Z+2, \\&\\
\overline{L(k,i)}^{\sigma_3,-}=&V_{\mathbb{Z}\gamma}^{T_2,+}\otimes W(k,i)^-\oplus V_{\mathbb{Z}\gamma}^{T_2,-}\otimes W(k,i)^+, \ k+2i\in 4\Z+2.
\end{array}
$$	
\end{lem}
\begin{lem}\label{l3.17}
Let $k\in 2\Z_+$, then 	as $V_{\Z\gamma}^+\otimes K_0^+$-modules, we have 
	\begin{equation}\label{e3.41}
	\overline{L(k,\frac{k}{2})}^{\sigma_3,1}=V^{T_1,+}\otimes W(k,\frac{k}{2})^+\oplus V^{T_1,-}\otimes W(k,\frac{k}{2})^-,
	\end{equation}
	\begin{equation}\label{e3.42}
	\overline{L(k,\frac{k}{2})}^{\sigma_3,2}=V^{T_1,+}\otimes \widetilde{W(k,\frac{k}{2})}^-\oplus V^{T_1,-}\otimes \widetilde{W(k,\frac{k}{2})}^+,
	\end{equation}
	\begin{equation}\label{e3.44}
	\overline{L(k,\frac{k}{2})}^{\sigma_3,3}=V^{T_1,+}\otimes \widetilde{W(k,\frac{k}{2})}^+\oplus V^{T_1,-}\otimes \widetilde{W(k,\frac{k}{2})}^-,
	\end{equation}
	\begin{equation}\label{e3.43}
	\overline{L(k,\frac{k}{2})}^{\sigma_3,4}=V^{T_1,-}\otimes W(k,\frac{k}{2})^+\oplus V^{T_1,+}\otimes W(k,\frac{k}{2})^-.
	\end{equation}	
	For $k\in 4\Z$, 
	\begin{equation}\label{e3.58}
		\overline{L(k,\frac{k}{2})}^{\sigma_2,1}\cong\overline{L(k,\frac{k}{2})}^{\sigma_3,1}, \quad 
			\overline{L(k,\frac{k}{2})}^{\sigma_2,2}\cong\overline{L(k,\frac{k}{2})}^{\sigma_3,2},
		\end{equation}
\begin{equation}\label{e3.59}		
			\overline{L(k,\frac{k}{2})}^{\sigma_2,3}\cong\overline{L(k,\frac{k}{2})}^{\sigma_3,4},  \quad 
				\overline{L(k,\frac{k}{2})}^{\sigma_2,4}\cong\overline{L(k,\frac{k}{2})}^{\sigma_3,3}.
	\end{equation}
For $k\in 4\Z+2$, 	
		\begin{equation}\label{e3.60}
	\overline{L(k,\frac{k}{2})}^{\sigma_2,1}=V^{T_2,+}\otimes W(k,\frac{k}{2})^+\oplus V^{T_2,-}\otimes W(k,\frac{k}{2})^-,
	\end{equation}
	\begin{equation}\label{e3.61}
	\overline{L(k,\frac{k}{2})}^{\sigma_2,2}=V^{T_2,+}\otimes \widetilde{W(k,\frac{k}{2})}^-\oplus V^{T_2,-}\otimes \widetilde{W(k,\frac{k}{2})}^+,
	\end{equation}
	\begin{equation}\label{e3.62}
	\overline{L(k,\frac{k}{2})}^{\sigma_2,3}=V^{T_2,-}\otimes W(k,\frac{k}{2})^+\oplus V^{T_2,+}\otimes W(k,\frac{k}{2})^-,
	\end{equation}	
	\begin{equation}\label{e3.63}
	\overline{L(k,\frac{k}{2})}^{\sigma_2,4}=V^{T_2,+}\otimes \widetilde{W(k,\frac{k}{2})}^+\oplus V^{T_2,-}\otimes \widetilde{W(k,\frac{k}{2})}^-.
	\end{equation}
\end{lem}
\begin{prop}\label{t3.1}
	All the irreducible modules of $L_{\widehat{\mathfrak{sl}_2}}(k,0)^{K}$ are self-dual.
	\end{prop}
\begin{proof}
	By Theorem 3.25 in \cite{JW2}, as irreducible modules of $L(k,0)^{\sigma_1}$,  $(L(k,i)^{\pm})'$ is isomorphic to $L(k,i)^{\mp}$, for $i\in 2\Z+1$. By Theorem \ref{main1},  as $L(k,0)^K$-modules, $L(k,i)^+$ and $L(k,i)^-$ are irreducible and isomorphic to each other, for $i\in 2\Z+1$. This means that in this case, $L(k,i)^+$ is self-dual. If $i\in 2\Z$, notice that \cite{JW2}
	$$
	V_{\Z\gamma}\otimes M^{i,\frac{i}{2}}\subseteq L(k,i)$$ and $V_{\Z\gamma}^+, \ V_{\Z\gamma}^-, 
	\ (M^{i,\frac{i}{2}})^+, \  (M^{i,\frac{i}{2}})^-
	$
	are self-dual. This deduces that $L(k,i)^j$, $j=1,2,3,4$ are self-dual.  We know from \cite{JW2} that 
		as $L(k,0)^{\sigma_{r},+}$-modules,
	$$
	(\overline{L(k,i)}^{\sigma_r,\pm})'\cong \overline{L(k,k-i)}^{\sigma_r,\pm}.
	$$
  It follows from  \cite{JWa}  that as $L(k,0)^K$-mouldes,
	$$
	\overline{L(k,i)}^{\sigma_r,\pm}\cong \overline{L(k,k-i)}^{\sigma_r,\pm}.
	$$
	Then we deduce that for $i\neq \frac{k}{2}$,  	$\overline{L(k,i)}^{\sigma_r,\pm}$, $r=1,2,3$ are self-dual. If $k\in 2\Z$,  $i=\frac{k}{2}$, by the fact that
	$$
	V_{\Z\gamma}\otimes M^{\frac{k}{2},0} \subseteq \overline{L(k,\frac{k}{2})}^{\sigma_1}
	$$
	and $V_{\Z\gamma}^{\pm}$, $(M^{\frac{k}{2},0})^{\pm}$ are self-dual, we deduce that $\overline{L(k,\frac{k}{2})}^{\sigma_1,j}$, $j=1,2,3,4$ are self-dual. This means that $\overline{L(k,\frac{k}{2})}^{\sigma_r,j}$, $r=2,3$, $j=1,2,3,4$ are self-dual also.
\end{proof}

The following proposition which gives quantum dimensions of all irreducible modules of $L(k,0)^K$  follows from \cite{JW2} and Theorem \ref{thm:orbifold3'}.
\begin{prop}\label{prop3.1} For $0\leq i\leq k$, we have
	$$
	{\rm dim}_qL(k,i)^+=2\dfrac{\sin\frac{(i+1)\pi}{k+2}}{\sin \frac{\pi}{k+2}}, \ i\in 2\Z+1, 	
	$$
	$$
	{\rm dim}_qL(k,i)^j=\dfrac{\sin\frac{(i+1)\pi}{k+2}}{\sin \frac{\pi}{k+2}}, \ i\in 2\Z, 	\  j=1,2,3,4,
	$$	
	$$
	{\rm dim}_q\overline{L(k,i)}^{\sigma_r,+}={\rm dim}_q\overline{L(k,i)}^{\sigma_r,-}=2\dfrac{\sin\frac{(i+1)\pi}{k+2}}{\sin \frac{\pi}{k+2}}, \ i\neq\frac{k}{2}, \ r=1,2,3,
	$$	
	$$
	{\rm dim}_qL(k, \frac{k}{2})^{\sigma_r,j}=\dfrac{1}{\sin \frac{\pi}{k+2}}, \ r=1,2,3, 	\  j=1,2,3,4. 
	$$	
\end{prop}
\vskip 0.3cm
For $0\leqslant i,j,l\leqslant k$ such that $i+j+l\in 2\mathbb{Z}$, following \cite{JW2}, we define
\[\mbox{sign}(i,j,l)^{+}=\begin{cases}+,\ & \mbox{if} \ i+j-l\in 4{\mathbb{Z}},\cr
-,\ &\mbox{if} \  i+j-l\notin 4{\mathbb{Z}},\end{cases}\]

and

\[\mbox{sign}(i,j,l)^{-}=\begin{cases}-,\ & \mbox{if} \ i+j-l\in 4{\mathbb{Z}},\cr
+,\ &\mbox{if} \  i+j-l\notin 4{\mathbb{Z}}.\end{cases}\]

We have the following fusion rules for the ${\Z}_{2}$-orbifold affine vertex operator algebra $L(k,0)^{\la\sigma_2\ra}$ from \cite{JW2}.

\begin{lem}\label{lem3.7} The fusion rules for the ${\mathbb{Z}}_{2}$-orbifold affine vertex operator algebra $L(k,0)^{\la \sigma_r \ra}$ are as follows:
\begin{eqnarray}\label{e3.21}
L(k,i)^{\sigma_2,+}\boxtimes L(k,j)^{\sigma_2,\pm}=\sum\limits_{\tiny{\begin{split}|i-j|\leqslant l\leqslant i+j \\  i+j+l\in 2\mathbb{Z} \ \ \ \\ i+j+l\leqslant 2k\ \ \ \end{split}}} L(k,l)^{\sigma_2, \mbox{sign}(i,j,l)^{\pm}},
\end{eqnarray}

\begin{eqnarray}\label{fusion.untwist2}
L(k,i)^{\sigma_2,-}\boxtimes L(k,j)^{\sigma_2,\pm}=\sum\limits_{\tiny{\begin{split}|i-j|\leqslant l\leqslant i+j \\  i+j+l\in 2\mathbb{Z} \ \ \ \\ i+j+l\leqslant 2k\ \ \ \end{split}}}  L(k,l)^{\sigma_2, \mbox{sign}(i,j,l)^{\mp}},
\end{eqnarray}

\begin{eqnarray}\label{fusion.twist1}
L(k,i)^{\sigma_2,+}\boxtimes \overline{L(k,j)}^{\sigma_2,\pm}=\sum\limits_{\tiny{\begin{split}|i-j|\leqslant l\leqslant i+j \\  i+j+l\in 2\mathbb{Z} \ \ \ \\ i+j+l\leqslant 2k\ \ \ \end{split}}} \overline{L(k,l)}^{\sigma_2, \mbox{sign}(i,j,l)^{\pm}},
\end{eqnarray}

\begin{eqnarray}\label{fusion.twist2}
L(k,i)^{\sigma_2,-}\boxtimes \overline{L(k,j)}^{\sigma_2,\pm}=\sum\limits_{\tiny{\begin{split}|i-j|\leqslant l\leqslant i+j \\  i+j+l\in 2\mathbb{Z} \ \ \ \\ i+j+l\leqslant 2k\ \ \ \end{split}}}  \overline{L(k,l)}^{\sigma_2, \mbox{sign}(i,j,l)^{\mp}}.
\end{eqnarray}
\end{lem}
We  also need the following relations which  re-correct some  typos of the original ones in Theorem 5.3 of \cite{JW2}.
\begin{lem}\label{l3.18}
	For $k\in 4\Z+2$, we have
	\begin{eqnarray}
	(M^{\frac{k}{2},0})^{+}\boxtimes W^{\pm}=\sum\limits_{\tiny{\begin{split}0\leq l\leq \frac{k}{2}-1 \\  k+l\in 2\mathbb{Z} \ \ \end{split}}}W(k,l)^{\pm}, \label{e11}
	\end{eqnarray}
	\begin{eqnarray}
	(M^{\frac{k}{2},0})^{-}\boxtimes W^{\pm}=\sum\limits_{\tiny{\begin{split}0\leq l\leq \frac{k}{2}-1 \\  k+l\in 2\mathbb{Z} \ \ \end{split}}} W(k,l)^{\mp},
	\label{e12}
	\end{eqnarray}
	where $W=W(k,\frac{k}{2})$ or $W=\widetilde{W(k,\frac{k}{2})}$.\\
	
	\vskip 0.1cm
	 For $k\in 4\mathbb{Z}$, we have
	\begin{eqnarray}\label{e13}
		(M^{\frac{k}{2},0})^{+}\boxtimes W(k,\frac{k}{2})^{\pm}=\sum\limits_{\tiny{\begin{split}0\leq l\leq \frac{k}{2}-1 \\ k+l\in 2\mathbb{Z} \ \ \ \ \ \end{split}}} W(k,l)^{\pm}+W(k,\frac{k}{2})^{\pm},
	\end{eqnarray}

	\begin{eqnarray}\label{e14}
		(M^{\frac{k}{2},0})^{-}\boxtimes W(k,\frac{k}{2})^{\pm}=\sum\limits_{\tiny{\begin{split}0\leq l\leq \frac{k}{2}-1 \\ k+l\in 2\mathbb{Z} \ \ \ \ \ \end{split}}}  W(k,l)^{\mp}+W(k,\frac{k}{2})^{\mp},
	\end{eqnarray}	
	\begin{eqnarray}\label{e15}
		(M^{\frac{k}{2},0})^{+}\boxtimes \widetilde{W(k,\frac{k}{2})}^{\pm}=\sum\limits_{\tiny{\begin{split}0\leq l\leq \frac{k}{2}-1 \\ k+l\in 2\mathbb{Z} \ \ \ \ \ \end{split}}} W(k,l)^{\pm}+\widetilde{W(k,\frac{k}{2})}^{\mp},
	\end{eqnarray}

	\begin{eqnarray}\label{e16}
		(M^{\frac{k}{2},0})^{-}\boxtimes \widetilde{W(k,\frac{k}{2})}^{\pm}=\sum\limits_{\tiny{\begin{split}0\leq l\leq \frac{k}{2}-1 \\ k+l\in 2\mathbb{Z}  \ \ \ \ \ \end{split}}}  W(k,l)^{\mp}+\widetilde{W(k,\frac{k}{2})}^{\pm}.
	\end{eqnarray}
	\end{lem}

\section{Fusion rules  of $L_{\widehat{\mathfrak{sl}_2}}(k,0)^{K}$}

In this section we  give  the fusion rules of irreducible modules of $L(k,0)^K$  through several theorems. 
\begin{thm}\label{thm3.8}
\begin{equation}\label{e3.25}
L(k,i)^{+}\boxtimes L(k,j)^{+}=\sum\limits_{\tiny{\begin{split}|i-j|\leqslant l\leqslant i+j \\  l\in 2\mathbb{Z}, i+j+l\leqslant 2k\ \ \ \end{split}}} L(k,l), \ i,j\in 2\Z+1,
\end{equation}
\begin{equation}\label{e3.26}
L(k,i)^{+}\boxtimes {L(k,j)}^{r}=\sum\limits_{\tiny{\begin{split}|i-j|\leqslant l\leqslant i+j \\  i+j+l\in 2\mathbb{Z} \ \ \ \\ i+j+l\leqslant 2k\ \ \ \end{split}}}  {L(k,l)}^{+}, \ i\in2\Z+1, \ j\in 2\Z, \ r=1,2,3,4.
\end{equation}
For $i,j\in2\Z$, we have
\begin{equation}\label{e3.27}
L(k,i)^{r}\boxtimes {L(k,j)}^{r}=\sum\limits_{\tiny{\begin{split}|i-j|\leqslant l\leqslant i+j \\  i+j+l\in 4\mathbb{Z} \ \ \ \\ i+j+l\leqslant 2k\ \ \ \end{split}}}  {L(k,l)}^{1}\oplus\sum\limits_{\tiny{\begin{split}|i-j|\leqslant l\leqslant i+j \\  i+j+l\in 4\mathbb{Z} +2\ \ \ \\ i+j+l\leqslant 2k\ \ \ \end{split}}}  {L(k,l)}^{4}, \quad r=1,2,3,4,
\end{equation}	
\begin{equation}\label{e3.28}
L(k,i)^{3}\boxtimes {L(k,j)}^{4}=L(k,i)^{2}\boxtimes {L(k,j)}^{1}=\sum\limits_{\tiny{\begin{split}|i-j|\leqslant l\leqslant i+j \\  i+j+l\in 4\mathbb{Z} \ \ \ \\ i+j+l\leqslant 2k\ \ \ \end{split}}}  {L(k,l)}^{2}\oplus\sum\limits_{\tiny{\begin{split}|i-j|\leqslant l\leqslant i+j \\ i+j+ l\in 4\mathbb{Z} +2\ \ \ \\ i+j+l\leqslant 2k\ \ \ \end{split}}}  {L(k,l)}^{3}, 
\end{equation}
\begin{equation}\label{e3.29}
L(k,i)^{2}\boxtimes {L(k,j)}^{4}=L(k,i)^{3}\boxtimes {L(k,j)}^{1}=\sum\limits_{\tiny{\begin{split}|i-j|\leqslant l\leqslant i+j \\  i+j+l\in 4\mathbb{Z} \ \ \ \\ i+j+l\leqslant 2k\ \ \ \end{split}}}  {L(k,l)}^{3}\oplus\sum\limits_{\tiny{\begin{split}|i-j|\leqslant l\leqslant i+j \\  i+j+l\in 4\mathbb{Z} +2\ \ \ \\ i+j+l\leqslant 2k\ \ \ \end{split}}}  {L(k,l)}^{2}, 
\end{equation}
\begin{equation}\label{e3.30}
L(k,i)^{2}\boxtimes {L(k,j)}^{3}=L(k,i)^{4}\boxtimes {L(k,j)}^{1}=\sum\limits_{\tiny{\begin{split}|i-j|\leqslant l\leqslant i+j \\  i+j+l\in 4\mathbb{Z} \ \ \ \\ i+j+l\leqslant 2k\ \ \ \end{split}}}  {L(k,l)}^{4}\oplus\sum\limits_{\tiny{\begin{split}|i-j|\leqslant l\leqslant i+j \\  i+j+l\in 4\mathbb{Z} +2\ \ \ \\ i+j+l\leqslant 2k\ \ \ \end{split}}}  {L(k,l)}^{1}.
\end{equation}
\end{thm}
\begin{proof} Notice that for $i\in 2\Z+1$, $L(k,i) ^+\cong L(k,i)^-$ as irreducible $L(k,0)^K$-modules. Then it is easy to see that  (\ref{e3.25})-(\ref{e3.26})  follow from Proposition \ref{prop3.1} and Lemma \ref{lem3.7}. 
For the rest formulas,	we first prove (\ref{e3.27}) for $r=1$. Since by Lemma \ref{l3.8}, for $i\in 2\Z$, 
$$
L(k,i)^1=L(k,i)^{\sigma_2,1}, \ L(k,i)^4=L(k,i)^{\sigma_2,4},
$$
it is enough to prove that for $i,j\in 2\Z$, 
$$
L(k,i)^{\sigma_2,1}\boxtimes {L(k,j)}^{\sigma_2,1}=\sum\limits_{\tiny{\begin{split}|i-j|\leqslant l\leqslant i+j \\  i+j+l\in 4\mathbb{Z} \ \ \ \\ i+j+l\leqslant 2k\ \ \ \end{split}}}  {L(k,l)}^{\sigma_2,1}\oplus\sum\limits_{\tiny{\begin{split}|i-j|\leqslant l\leqslant i+j \\ i+j+ l\in 4\mathbb{Z} +2\ \ \ \\ i+j+l\leqslant 2k\ \ \ \end{split}}}  {L(k,l)}^{\sigma_2,4}.
$$
	Recall that 
	$$
	L(k,i)=\bigoplus_{j=0}^{k-1}V_{\mathbb{Z}\gamma+(i-2j)\gamma/2k}\otimes M^{i,j} \ \ \ \mbox{for} \ 0\leq i\leq k,
	$$
	and for $i\in 2\Z$, 
	$M^{i,\frac{i}{2}}$ is generated by $v^{i,\frac{i}{2}}$. Then for $i\in 4\Z$, 
	\begin{eqnarray}\label{e3.33}
	V_{\mathbb{Z}\gamma}^+\otimes (M^{i,\frac{i}{2}})^+\subseteq L(k,i)^{\sigma_2,1}, \quad 
	V_{\mathbb{Z}\gamma}^+\otimes (M^{i,\frac{i}{2}})^-\subseteq L(k,i)^{\sigma_2,3} 
	\end{eqnarray}
	 and for 
	$i\in 4\Z+2$, 
	\begin{eqnarray}\label{e3.34}
	V_{\mathbb{Z}\gamma}^+\otimes (M^{i,\frac{i}{2}})^+\subseteq L(k,i)^{\sigma_2,2}, \quad V_{\mathbb{Z}\gamma}^+\otimes (M^{i,\frac{i}{2}})^-\subseteq L(k,i)^{\sigma_2,4}.
	\end{eqnarray}
	By (\ref{e3.21}), we have 
	 \begin{eqnarray}\label{e3.32}
	L(k,i)^{\sigma_2,+}\boxtimes L(k,j)^{\sigma_2,+}=\sum\limits_{\tiny{\begin{split}|i-j|\leq l\leq i+j \\  i+j+l\in 4\mathbb{Z} \ \ \ \\ i+j+l\leq 2k\ \ \ \end{split}}} L(k,l)^{\sigma_2,+}\oplus 
	\sum\limits_{\tiny{\begin{split}|i-j|\leq l\leq i+j \\  i+j+l\in 4\mathbb{Z}+2\ \ \ \\ i+j+l\leq 2k\ \ \ \end{split}}}
	L(k,l)^{\sigma_2,-}.
	\end{eqnarray}
	 By Theorem 5.1 of \cite{JW2}, we have for $i,j\in 2\Z_+$, 
	 \begin{eqnarray}
	 (M^{i,\frac{i}{2}})^+\boxtimes (M^{j,\frac{j}{2}})^+=\sum\limits_{\tiny{\begin{split}|i-j|\leq l\leq i+j \\  i+j+l\in 4\mathbb{Z} \ \ \
	  \\ i+j+l\leq 2k\ \ \ \end{split}}} (M^{l,\frac{l}{2}})^+\oplus 
	\sum\limits_{\tiny{\begin{split}|i-j|\leq l\leq i+j \\  i+j+l\in 4\mathbb{Z}+2\ \ \ \\ i+j+l\leq 2k\ \ \ \end{split}}} (M^{l,\frac{l}{2}})^-.\label{eq:5.31}
	\end{eqnarray}
	So
	 \begin{eqnarray}
	\begin{split}&(V_{\Z\gamma}^+\otimes (M^{i,\frac{i}{2}})^+)\boxtimes(V_{\Z\gamma}^+\otimes (M^{j,\frac{j}{2}})^+)\\
	&=\sum\limits_{\tiny{\begin{split}|i-j|\leq l\leq i+j \\  i+
			j+l\in 4\mathbb{Z} \ \ \
			\\ i+j+l\leq 2k\ \ \ \end{split}}} V_{\Z\gamma}^+\otimes (M^{l,\frac{l}{2}})^+\oplus 
	\sum\limits_{\tiny{\begin{split}|i-j|\leq l\leq i+j \\  i+j+l\in 4\mathbb{Z}+2\ \ \ \\ i+j+l\leq 2k\ \ \ \end{split}}}V_{\Z\gamma}^+\otimes (M^{l,\frac{l}{2}})^-.\label{e101}
	\end{split}
	\end{eqnarray}
	Notice that for $0\leq l\leq k$, 
	$$L(k,l)^{\sigma_2,+}=L(k,l)^{\sigma_2,1}\oplus L(k,l)^{\sigma_2,2}, \quad L(k,l)^{\sigma_2,-}=L(k,l)^{\sigma_2,3}\oplus L(k,l)^{\sigma_2,4}.$$ 
Then  by (\ref{e3.33})-(\ref{e101}),  if $i,j\in 4\Z$, 
$$
L(k,i)^{\sigma_2,1}\boxtimes {L(k,j)}^{\sigma_2,1}=\sum\limits_{\tiny{\begin{split}|i-j|\leqslant l\leqslant i+j \\  i+j+l\in 4\mathbb{Z} \ \ \ \\ i+j+l\leqslant 2k\ \ \ \end{split}}}  {L(k,l)}^{\sigma_2,1}\oplus\sum\limits_{\tiny{\begin{split}|i-j|\leqslant l\leqslant i+j \\ i+j+ l\in 4\mathbb{Z} +2\ \ \ \\ i+j+l\leqslant 2k\ \ \ \end{split}}}  {L(k,l)}^{\sigma_2,4}.
$$
If $i\in 4\Z$, $j\in 4\Z+2$, then
$$
L(k,i)^{\sigma_2,1}\boxtimes {L(k,j)}^{\sigma_2,2}=\sum\limits_{\tiny{\begin{split}|i-j|\leqslant l\leqslant i+j \\  i+j+l\in 4\mathbb{Z} \ \ \ \\ i+j+l\leqslant 2k\ \ \ \end{split}}}  {L(k,l)}^{\sigma_2,2}\oplus\sum\limits_{\tiny{\begin{split}|i-j|\leqslant l\leqslant i+j \\ i+j+ l\in 4\mathbb{Z} +2\ \ \ \\ i+j+l\leqslant 2k\ \ \ \end{split}}}  {L(k,l)}^{\sigma_2,3}.
$$
If $i,j\in 4\Z+2$, then 
$$
L(k,i)^{\sigma_2,2}\boxtimes {L(k,j)}^{\sigma_2,2}=\sum\limits_{\tiny{\begin{split}|i-j|\leqslant l\leqslant i+j \\  i+j+l\in 4\mathbb{Z} \ \ \ \\ i+j+l\leqslant 2k\ \ \ \end{split}}}  {L(k,l)}^{\sigma_2,1}\oplus\sum\limits_{\tiny{\begin{split}|i-j|\leqslant l\leqslant i+j \\ i+j+ l\in 4\mathbb{Z} +2\ \ \ \\ i+j+l\leqslant 2k\ \ \ \end{split}}}  {L(k,l)}^{\sigma_2,4}.
$$
Since by Lemma \ref{l3.7}, 
$$
L(k,0)^{\sigma_2, 2}\boxtimes L(k,i)^{\sigma_2,1}=L(k,i)^{\sigma_2,2}, \ L(k,0)^{\sigma_2, 2}\boxtimes L(k,i)^{\sigma_2,3}=L(k,i)^{\sigma_2,4},
$$
and 
$$
L(k,0)^{\sigma_2, r}\boxtimes L(k,0)^{\sigma_2,r}=L(k,0)^{\sigma_2,1}, \ r=1,2,3,4,
$$
it follows that for $i,j\in 2\Z$, 
$$
L(k,i)^{\sigma_2,1}\boxtimes {L(k,j)}^{\sigma_2,1}=\sum\limits_{\tiny{\begin{split}|i-j|\leqslant l\leqslant i+j \\  i+j+l\in 4\mathbb{Z} \ \ \ \\ i+j+l\leqslant 2k\ \ \ \end{split}}}  {L(k,l)}^{\sigma_2,1}\oplus\sum\limits_{\tiny{\begin{split}|i-j|\leqslant l\leqslant i+j \\ i+j+ l\in 4\mathbb{Z} +2\ \ \ \\ i+j+l\leqslant 2k\ \ \ \end{split}}}  {L(k,l)}^{\sigma_2,4}.
$$
We show that  (\ref{e3.27}) is true  for $r=1$. (\ref{e3.27}) for $r=2,3,4$ and 
(\ref{e3.28})-(\ref{e3.30})  then follow from (\ref{e3.27}) for $r=1$, Lemma \ref{l3.7},  and the associativity of fusion product.
\end{proof}
\begin{thm}\label{thm3.9}
	 For $i\in 2\Z_++1$, $0\leq j<\frac{k}{2}$,  $r=1,2,3$,
\begin{equation}\label{e3.35}
L(k,i)^{+}\boxtimes \overline{L(k,j)}^{\sigma_r,+}=L(k,i)^{+}\boxtimes \overline{L(k,j)}^{\sigma_r,-}
=\sum\limits_{\tiny{\begin{split}|i-j|\leq l\leq i+j \\  i+j+l\in 2\mathbb{Z} \ \ \
		\\ i+j+l\leq 2k\ \ \ \end{split}}}(\overline{L(k,l)}^{\sigma_r,+}+\overline{L(k,l)}^{\sigma_r,-}).
\end{equation}
For  $k\in 2\Z_+$, $i\in 2\Z_++1$, $r=1,2,3$, $j=1,2,3,4$, 
\begin{equation}\label{e3.36}
	L(k,i)^{+}\boxtimes \overline{L(k,\frac{k}{2})}^{\sigma_r,j}
	=\sum\limits_{\tiny{\begin{split}|i-j|\leq l\leq i+j \\  i+j+l\in 2\mathbb{Z} \ \ \
			\\ l\leq \frac{k}{2}-1\ \ \ \end{split}}}(\overline{L(k,l)}^{\sigma_r,+}+\overline{L(k,l)}^{\sigma_r,-}).
\end{equation}
		For $i\in 2\Z_+$, $0\leq j< \frac{k}{2}$, 
	$$
	L(k,i)^{1}\boxtimes \overline{L(k,j)}^{\sigma_1,\pm}=L(k,i)^{2}\boxtimes \overline{L(k,j)}^{\sigma_1,\pm}
	=\sum\limits_{\tiny{\begin{split}|i-j|\leq l\leq i+j \\  i+j+l\in 2\mathbb{Z} \ \ \
			\\ i+j+l\leq 2k\ \ \ \end{split}}}\overline{L(k,l)}^{\sigma_1,sign(i,j,l)^{\pm}},$$			
	$$
L(k,i)^{3}\boxtimes \overline{L(k,j)}^{\sigma_1,\pm}=L(k,i)^{4}\boxtimes \overline{L(k,j)}^{\sigma_1,\pm}
=\sum\limits_{\tiny{\begin{split}|i-j|\leq l\leq i+j \\  i+j+l\in 2\mathbb{Z} \ \ \
		\\ i+j+l\leq 2k\ \ \ \end{split}}}\overline{L(k,l)}^{\sigma_1,sign(i,j,l)^{\mp}},$$	
$$
L(k,i)^{1}\boxtimes \overline{L(k,j)}^{\sigma_2,\pm}=L(k,i)^{3}\boxtimes \overline{L(k,j)}^{\sigma_2,\pm}
=\sum\limits_{\tiny{\begin{split}|i-j|\leq l\leq i+j \\  i+j+l\in 2\mathbb{Z} \ \ \
		\\ i+j+l\leq 2k\ \ \ \end{split}}}\overline{L(k,l)}^{\sigma_2,sign(i,j,l)^{\pm}},$$	
	$$
	L(k,i)^{2}\boxtimes \overline{L(k,j)}^{\sigma_2,\pm}=L(k,i)^{4}\boxtimes \overline{L(k,j)}^{\sigma_2,\pm}
	=\sum\limits_{\tiny{\begin{split}|i-j|\leq l\leq i+j \\  i+j+l\in 2\mathbb{Z} \ \ \
			\\ i+j+l\leq 2k\ \ \ \end{split}}}\overline{L(k,l)}^{\sigma_2,sign(i,j,l)^{\mp}},$$		
$$
L(k,i)^{2}\boxtimes \overline{L(k,j)}^{\sigma_3,\pm}=L(k,i)^{3}\boxtimes \overline{L(k,j)}^{\sigma_3,\pm}
=\sum\limits_{\tiny{\begin{split}|i-j|\leq l\leq i+j \\  j-l\in 4\mathbb{Z}+2 \ \ \
		\\ i+j+l\leq 2k\ \ \ \end{split}}}\overline{L(k,l)}^{\sigma_3,\pm}\oplus 
	\sum\limits_{\tiny{\begin{split}|i-j|\leq l\leq i+j \\  j-l\in 4\mathbb{Z} \ \ \
			\\ i+j+l\leq 2k\ \ \ \end{split}}}\overline{L(k,l)}^{\sigma_3,\mp},$$	
		$$
		L(k,i)^{1}\boxtimes \overline{L(k,j)}^{\sigma_3,\pm}=L(k,i)^{4}\boxtimes \overline{L(k,j)}^{\sigma_3,\pm}
		=\sum\limits_{\tiny{\begin{split}|i-j|\leq l\leq i+j \\  j-l\in 4\mathbb{Z}+2 \ \ \
				\\ i+j+l\leq 2k\ \ \ \end{split}}}\overline{L(k,l)}^{\sigma_3,\mp}\oplus 
		\sum\limits_{\tiny{\begin{split}|i-j|\leq l\leq i+j \\  j-l\in 4\mathbb{Z} \ \ \
				\\ i+j+l\leq 2k\ \ \ \end{split}}}\overline{L(k,l)}^{\sigma_3,\pm}.$$			
\end{thm}
\begin{proof}  (\ref{e3.35})-(\ref{e3.36}) follow from Proposition \ref{prop3.1}, Lemma \ref{lem3.7}, and the fact that $L(k,i)^+\cong L(k,i)^-$, if $i\in 2\Z_++1$.  For the rest formulas, recall that for $i\in 2\Z_+$,  $r=1,2,3$, 
	$$
	L(k,i)^{\sigma_r,1}\oplus L(k,i)^{\sigma_r,2}=L(k,i)^{\sigma_r,+}, \quad 
L(k,i)^{\sigma_r,3}\oplus L(k,i)^{\sigma_r,4}=L(k,i)^{\sigma_r,-}.
	$$
Then these formulas follow from Lemma \ref{lem3.7}, Proposition \ref{prop3.1}, and Lemma \ref{l3.8}.
	\end{proof}
 For $i\in 2\Z_+$, we denote $\bar{i}=2$ if $i\in 4\Z+2$, and $\bar{i}=0$ if $i\in 4\Z$.
\begin{thm} \label{thm3.10}
	For $k\in 2\Z_+$, $i\in 2\Z_+$, $r=1,2,3$, 
\begin{eqnarray}
\begin{split}&L(k,i)^{\sigma_r,j}\boxtimes \overline{L(k,\frac{k}{2})}^{\sigma_r,1}\\&\\&
=\sum\limits_{\tiny{\begin{split}|i-\frac{k}{2}|\leq l\leq \frac{k}{2}-1 
		\\ i+\frac{k}{2}-l\in 4\Z\ \ \ \end{split}}}\overline{L(k,l)}^{\sigma_r, +}\oplus \sum\limits_{\tiny{\begin{split}|i-\frac{k}{2}|\leq l\leq \frac{k}{2}-1 
		\\ i+\frac{k}{2}-l\in 4\Z+2\ \ \ \end{split}}}\overline{L(k,l)}^{\sigma_r, -}\oplus 
\overline{L(k,\frac{k}{2})}^{\sigma_r,j+\bar{i}}, \ j=1,2, \label{eq105}
\end{split}
\end{eqnarray}
\begin{eqnarray}\begin{split}
&L(k,i)^{\sigma_r, j}\boxtimes \overline{L(k,\frac{k}{2})}^{\sigma_r,1}\\&\\&
=\sum\limits_{\tiny{\begin{split}|i-\frac{k}{2}|\leq l\leq \frac{k}{2}-1 
		\\ i+\frac{k}{2}-l\in 4\Z\ \ \ \end{split}}}\overline{L(k,l)}^{\sigma_r, -}\oplus \sum\limits_{\tiny{\begin{split}|i-\frac{k}{2}|\leq l\leq \frac{k}{2}-1 
		\\i+ \frac{k}{2}-l\in 4\Z+2\ \ \ \end{split}}}\overline{L(k,l)}^{\sigma_r, +}\oplus 
\overline{L(k,\frac{k}{2})}^{\sigma_r,j-\bar{i}}, \ j=3,4,\label{eq106}
\end{split}
\end{eqnarray}
$$
L(k,i)^{\sigma_r, j}\boxtimes \overline{L(k,\frac{k}{2})}^{\sigma_r,1}=L(k,i)^{\sigma_r, 1}\boxtimes \overline{L(k,\frac{k}{2})}^{\sigma_r,j}, \  j=1,2,3,4,
$$
$$
L(k,i)^{\sigma_r,j}\boxtimes \overline{L(k,\frac{k}{2})}^{\sigma_r,j}=L(k,i)^{\sigma_r,1}\boxtimes \overline{L(k,\frac{k}{2})}^{\sigma_r,1}, \  j=1,2,3,4,
$$
$$
L(k,i)^{\sigma_r, l}\boxtimes \overline{L(k,\frac{k}{2})}^{\sigma_r,s}=L(k,i)^{\sigma_r,t}\boxtimes \overline{L(k,\frac{k}{2})}^{\sigma_r,1}, \  \{l,s,t\}=\{2,3,4\}.
$$
\end{thm}
\begin{proof} It is enough to prove the theorem for $r=1$.
Notice that for $i\in 4\Z+2$, 
	\begin{equation}\label{eq101}
	V_{\Z\gamma}^+\otimes (M^{i,\frac{i}{2}})^+\oplus V_{\Z\gamma}^-\otimes (M^{i,\frac{i}{2}})^-\subseteq L(k,i)^3,  \quad V_{\Z\gamma}^+\otimes (M^{i,\frac{i}{2}})^-\oplus V_{\Z\gamma}^-\otimes (M^{i,\frac{i}{2}})^+\subseteq L(k,i)^4.
	\end{equation}
If 	$i\in 4\Z$, then
\begin{equation}\label{eq102}
V_{\Z\gamma}^+\otimes (M^{i,\frac{i}{2}})^+\oplus V_{\Z\gamma}^-\otimes (M^{i,\frac{i}{2}})^-\subseteq L(k,i)^1,  \quad V_{\Z\gamma}^+\otimes (M^{i,\frac{i}{2}})^-\oplus V_{\Z\gamma}^-\otimes (M^{i,\frac{i}{2}})^+\subseteq L(k,i)^2.
\end{equation}
By (\ref{eq3.7}), we have 
$$
\overline{L(k,\frac{k}{2})}^{\sigma_1}=\oplus_{j=0}^{k-1}V_{\mathbb{Z}\gamma-\frac{j}{k}\gamma}\otimes M^{\frac{k}{2},j},
$$
and
\begin{equation}\label{eq103}
V_{\Z\gamma}^+\otimes (M^{\frac{k}{2},0})^+\oplus V_{\Z\gamma}^-\otimes (M^{\frac{k}{2},0})^-\subseteq \overline{L(k,\frac{k}{2})}^{\sigma_1,1}, \ 
V_{\Z\gamma}^+\otimes (M^{\frac{k}{2},0})^-\oplus V_{\Z\gamma}^-\otimes (M^{\frac{k}{2},0})^+\subseteq \overline{L(k,\frac{k}{2})}^{\sigma_1,2}.
\end{equation}
By Theorem 5.1 of \cite{JW2}, 
\begin{equation}\label{eq104}
(M^{i,\frac{i}{2}})^{+}\boxtimes (M^{\frac{k}{2},0})^{\pm}=\sum\limits_{\tiny{\begin{split}|\frac{k}{2}-i|\leq l<\frac{k}{2} \\  i+j+l\in 2\mathbb{Z} \ \ \ \\ i+j+l\leq 2k\ \ \ \end{split}}} M^{l,\overline{(\frac{2l-k}{4}})}+(M^{\frac{k}{2},0})^{\pm},
\end{equation}
where $\overline{a}$ means the residue of the integer $a$ modulo $k$ for $0\leq a\leq k$. Then by (\ref{eq101}), (\ref{eq103})-(\ref{eq104}) and Proposition \ref{prop3.1}, we have for $i\in 4\Z+2$, 
$$
L(k,i)^{j}\boxtimes \overline{L(k,\frac{k}{2})}^{\sigma_1,1}=\sum\limits_{\tiny{\begin{split}0\leq l\leq \frac{k}{2}-1 
		\\ \frac{k}{2}+l\in 2\Z\ \ \ \end{split}}}\overline{L(k,l)}^{\sigma_1, sign(i,\frac{k}{2},l)^+}\oplus \overline{L(k,\frac{k}{2})}^{\sigma_1,j-2}, \ j=3,4.
$$
Then
$$
\begin{array}{ll}
&
L(k,i)^{1}\boxtimes \overline{L(k,\frac{k}{2})}^{\sigma_1,1}
=L(k,0)^3\boxtimes L(k,i)^{3}\boxtimes \overline{L(k,\frac{k}{2})}^{\sigma_1,1}
\\&\\&=\sum\limits_{\tiny{\begin{split}0\leq l\leq \frac{k}{2}-1 
		\\ \frac{k}{2}+l\in 2\Z\ \ \ \end{split}}}L(k,0)^3\boxtimes \overline{L(k,l)}^{\sigma_1, sign(i,\frac{k}{2},l)^+}\oplus L(k,0)^3\boxtimes \overline{L(k,\frac{k}{2})}^{\sigma_1,1}\\&\\&
	=\sum\limits_{\tiny{\begin{split}0\leq l\leq \frac{k}{2}-1 
				\\ \frac{k}{2}+l\in 2\Z\ \ \ \end{split}}}\overline{L(k,l)}^{\sigma_1, sign(i,\frac{k}{2},l)^-}\oplus \overline{L(k,\frac{k}{2})}^{\sigma_1,3}, 
\end{array}	 
$$	
and
$$
\begin{array}{ll}
&
L(k,i)^{2}\boxtimes \overline{L(k,\frac{k}{2})}^{\sigma_1,1}
=L(k,0)^4\boxtimes L(k,i)^{3}\boxtimes \overline{L(k,\frac{k}{2})}^{\sigma_1,1}
\\&\\&=\sum\limits_{\tiny{\begin{split}0\leq l\leq \frac{k}{2}-1 
		\\ \frac{k}{2}+l\in 2\Z\ \ \ \end{split}}}L(k,0)^4\boxtimes \overline{L(k,l)}^{\sigma_1, sign(i,\frac{k}{2},l)^+}\oplus L(k,0)^4\boxtimes \overline{L(k,\frac{k}{2})}^{\sigma_1,1}\\&\\&
=\sum\limits_{\tiny{\begin{split}0\leq l\leq \frac{k}{2}-1 
		\\ \frac{k}{2}+l\in 2\Z\ \ \ \end{split}}}\overline{L(k,l)}^{\sigma_1, sign(i,\frac{k}{2},l)^-}\oplus \overline{L(k,\frac{k}{2})}^{\sigma_1,4}.
\end{array}	 
$$	
For $i\in 4\Z$, by (\ref{eq102}), (\ref{eq103})-(\ref{eq104}) and Proposition \ref{prop3.1}, we have
$$
	L(k,i)^{j}\boxtimes \overline{L(k,\frac{k}{2})}^{\sigma_1,1}=\sum\limits_{\tiny{\begin{split}0\leq l\leq \frac{k}{2}-1 
			\\ \frac{k}{2}+l\in 2\Z\ \ \ \end{split}}}\overline{L(k,l)}^{\sigma_1, sign(i,\frac{k}{2},l)^+}\oplus \overline{L(k,\frac{k}{2})}^{\sigma_1,j}, \quad j=1,2.
$$
Then 
$$
L(k,i)^{j}\boxtimes \overline{L(k,\frac{k}{2})}^{\sigma_1,1}=\sum\limits_{\tiny{\begin{split}0\leq l\leq \frac{k}{2}-1 
		\\ \frac{k}{2}+l\in 2\Z\ \ \ \end{split}}}\overline{L(k,l)}^{\sigma_1, sign(i,\frac{k}{2},l)^-}\oplus \overline{L(k,\frac{k}{2})}^{\sigma_1,j}, \quad j=3,4.
$$
We prove (\ref{eq105}) and (\ref{eq106}). The last three relations follow from (\ref{e1})-(\ref{e3}) and Lemma \ref{l3.7}.
	\end{proof}

\begin{thm}\label{thm3.11}
	For $0\leq i,j<\frac{k}{2}$,  $a,b\in\{+,-\}$, and  $\{r,s,t\}=\{1,2,3\}$, 
\begin{equation}\label{e3.37}
\overline{L(k,i)}^{\sigma_r,a}\boxtimes \overline{L(k,j)}^{\sigma_s,b}=\sum\limits_{\tiny{\begin{split}|i-j|\leqslant l\leqslant i+j \\  i+j+l\in 2\mathbb{Z} \ \ \ \\ i+j+l\leqslant 2k\ \ \ \end{split}}} (\overline{L(k,l)}^{\sigma_t,+}+\overline{L(k,l)}^{\sigma_t,-}).
\end{equation}
For $k\in 2\Z$, $i\in 2\Z+1$, $i\neq \frac{k}{2}$, $j=1,2,3,4,$   $\{r,s,t\}=\{1,2,3\}$,
\begin{equation}\label{e3.38}
\overline{L(k,i)}^{\sigma_r,\pm}\boxtimes \overline{L(k,\frac{k}{2})}^{\sigma_s,j}=\sum\limits_{\tiny{\begin{split}0\leqslant l\leqslant \frac{k}{2}-1 \\  i+\frac{k}{2}+l\in 2\mathbb{Z} \ \ \ \end{split}}} (\overline{L(k,l)}^{\sigma_t}+\overline{L(k,l)}^{\sigma_t,-}).
\end{equation}
For $k\in 2\Z$, $i\in 2\Z$, $i\neq \frac{k}{2}$,	
\begin{equation}\label{e3.45}
\overline{L(k,i)}^{\sigma_1,+}\boxtimes \overline{L(k,\frac{k}{2})}^{\sigma_2,1}=\sum\limits_{\tiny{\begin{split}0\leqslant l\leqslant \frac{k}{2}-1 \\   i+\frac{k}{2}+l\in 2\mathbb{Z} \ \ \ \end{split}}} (\overline{L(k,l)}^{\sigma_3,+}+\overline{L(k,l)}^{\sigma_3,-})\oplus (\overline{L(k,\frac{k}{2})}^{\sigma_3, 1}\oplus\overline{L(k,\frac{k}{2})}^{\sigma_3, 4}), 
\end{equation}
\begin{equation}\label{e3.46}
\overline{L(k,i)}^{\sigma_1,-}\boxtimes \overline{L(k,\frac{k}{2})}^{\sigma_2,1}=\sum\limits_{\tiny{\begin{split}0\leqslant l\leqslant \frac{k}{2}-1 \\   i+\frac{k}{2}+l\in 2\mathbb{Z} \ \ \ \end{split}}} (\overline{L(k,l)}^{\sigma_3}+\overline{L(k,l)}^{\sigma_3,-})\oplus (\overline{L(k,\frac{k}{2})}^{\sigma_3, 2}\oplus\overline{L(k,\frac{k}{2})}^{\sigma_3, 3}), 
\end{equation}	
\vskip 0.2cm
\begin{equation}\label{e3.47}
\overline{L(k,i)}^{\sigma_1,\pm}\boxtimes \overline{L(k,\frac{k}{2})}^{\sigma_2,2}
=\overline{L(k,i)}^{\sigma_1,\mp}\boxtimes \overline{L(k,\frac{k}{2})}^{\sigma_2,1}, 
\end{equation}\begin{equation}\label{e3.64}
\overline{L(k,i)}^{\sigma_1,\pm}\boxtimes \overline{L(k,\frac{k}{2})}^{\sigma_2,3}
=\overline{L(k,i)}^{\sigma_1,\pm}\boxtimes \overline{L(k,\frac{k}{2})}^{\sigma_2,1},
\end{equation}
\begin{equation}\label{e3.48}
\overline{L(k,i)}^{\sigma_1,\pm}\boxtimes \overline{L(k,\frac{k}{2})}^{\sigma_2,4}
=\overline{L(k,i)}^{\sigma_1,\mp}\boxtimes \overline{L(k,\frac{k}{2})}^{\sigma_2,1}.
\end{equation}
	\end{thm}
\begin{proof}  We first prove (\ref{e3.37}). We may assume that  $r=1$, $s=2$, $t=3$. 
	Notice that by Lemma \ref{l3.7} and Theorem \ref{thm3.9},
$$
\begin{array}{ll}
&\overline{L(k,i)}^{\sigma_1,+}\boxtimes \overline{L(k,j)}^{\sigma_2,+}=L(k,0)^2\boxtimes \overline{L(k,i)}^{\sigma_1,+}\boxtimes \overline{L(k,j)}^{\sigma_2,+}
\\&\\&
=L(k,0)^{\sigma_2,3}\boxtimes \overline{L(k,i)}^{\sigma_1,+}\boxtimes \overline{L(k,j)}^{\sigma_2,+}\\&\\&
=\overline{L(k,i)}^{\sigma_1,+}\boxtimes \overline{L(k,j)}^{\sigma_2,-}.
\end{array}
$$	
Similarly
$$
\overline{L(k,i)}^{\sigma_1, -}\boxtimes \overline{L(k,j)}^{\sigma_2,\pm}
=\overline{L(k,i)}^{\sigma_1,+}\boxtimes \overline{L(k,j)}^{\sigma_2,+}.
$$	
On the other hand, by Propostion \ref{prop3.1} we have
$$
	{\rm dim}_q\overline{L(k,i)}^{\sigma_r,+}={\rm dim}_q\overline{L(k,i)}^{\sigma_r,-}=2\dfrac{\sin\frac{(i+1)\pi}{k+2}}{\sin \frac{\pi}{k+2}}, \ i\neq\frac{k}{2}, \ r=1,2,3.
$$	
Then 
$$
	{\rm dim}_q\overline{L(k,0)}^{\sigma_r,\pm}=2, \ r=1,2,3,
$$
and 
$$
{\rm dim}_q(\overline{L(k,0)}^{\sigma_1, +}\boxtimes \overline{L(k,j)}^{\sigma_2,\pm})=2\dfrac{\sin\frac{(j+1)\pi}{k+2}}{\sin \frac{\pi}{k+2}}.
$$
If $k\in 2\Z_+$, then by Lemma \ref{l3.16} and Lemma \ref{l3.17}, we have 
$$
\overline{L(k,0)}^{\sigma_1, +}\boxtimes \overline{L(k,j)}^{\sigma_2,+}
=\overline{L(k,j)}^{\sigma_3, +}\oplus \overline{L(k,j)}^{\sigma_3, -}.
$$
If $k\in 2\Z_++1$,  let $\beta=2\gamma$. Then  $V_{\Z\beta}^+\otimes K_0\subseteq L(k,0)^{K}$,  and as $V_{\Z\beta}^+\otimes K_0^+$-modules, 
$$
V_{\Z\beta+\frac{1}{8}\beta}\otimes K_0\subseteq \overline{L(k,0)}^{\sigma_1, +}, $$
$$
V_{\Z\beta}^{T_1,+}\otimes W(k,j)^+ \ {\rm or} \ V_{\Z\beta}^{T_2,+}\otimes W(k,j)^+\subseteq \overline{L(k,0)}^{\sigma_2, +},
$$
$$
V_{\Z\beta}^{T_1,+}\otimes W(k,j)^+ \ {\rm or} \ V_{\Z\beta}^{T_2,+}\otimes W(k,j)^+\subseteq \overline{L(k,0)}^{\sigma_3, +}.
$$
Then we also have 
$$
\overline{L(k,0)}^{\sigma_1, +}\boxtimes \overline{L(k,j)}^{\sigma_2,\pm}
=\overline{L(k,0)}^{\sigma_1, -}\boxtimes \overline{L(k,j)}^{\sigma_2,\pm}=\overline{L(k,j)}^{\sigma_3, +}\oplus \overline{L(k,j)}^{\sigma_3, -}.
$$
By considering $L(k,i)^+\boxtimes \overline{L(k,0)}^{\sigma_1, +}\boxtimes \overline{L(k,j)}^{\sigma_2,\pm}$ and using (\ref{e3.35}), we deduce that for $0\leq i,j<\frac{k}{2}$, 
$$
\overline{L(k,i)}^{\sigma_1,a}\boxtimes \overline{L(k,j)}^{\sigma_2,b}=\sum\limits_{\tiny{\begin{split}|i-j|\leqslant l\leqslant i+j \\  i+j+l\in 2\mathbb{Z} \ \ \ \\ i+j+l\leqslant 2k\ \ \ \end{split}}}( \overline{L(k,l)}^{\sigma_3,+}+\overline{L(k,l)}^{\sigma_3,-}), \ a,b\in\{+,-\}.
$$

\noindent
For (\ref{e3.38}),  it is enough to prove that
\begin{equation}\label{e3.40}\overline{L(k,i)}^{\sigma_1,\pm}\boxtimes \overline{L(k,\frac{k}{2})}^{\sigma_2,j}=\sum\limits_{\tiny{\begin{split}0\leqslant l\leqslant \frac{k}{2}-1 \\  i+\frac{k}{2}+l\in 2\mathbb{Z} \ \ \ \end{split}}} (\overline{L(k,l)}^{\sigma_3, +}\oplus\overline{L(k,l)}^{\sigma_3, -}), \quad j=1,2,3,4.
\end{equation}	
	Recall from  Lemma \ref{l3.16} that for  $k\in 2\Z$, $ 0\leq i\leq k$,
	$$
\overline{L(k,i)}^{\sigma_1,+}=\bigoplus_{j\in 2\Z, 0\leq j\leq k-1}V_{\mathbb{Z}\gamma+(\frac{k-2i}{4k}-\frac{j}{k})\gamma}\otimes M^{i,j}, $$$$ \overline{L(k,i)}^{\sigma_1,-}=\bigoplus_{j\in 2\Z+1, 0\leq j\leq k-1}V_{\mathbb{Z}\gamma+(\frac{k-2i}{4k}-\frac{j}{k})\gamma}\otimes M^{i,j}.
$$
By the fusion rules for $V_{\Z\gamma}^+$ \cite{Ab01}, we have
$$
V_{\mathbb{Z}\gamma+(\frac{k-2i}{4k}-\frac{j}{k})\gamma}\boxtimes V^{T_r,\pm}=V^{T_r,+}\oplus V^{T_r,-},  \quad {\rm if} \ \frac{k}{2}-i \ {\rm is \ even}, \quad r=1,2.
$$
$$
V_{\mathbb{Z}\gamma+(\frac{k-2i}{4k}-\frac{j}{k})\gamma}\boxtimes V^{T_1,\pm}=V^{T_2,+}\oplus V^{T_2,-},  \quad V_{\mathbb{Z}\gamma+(\frac{k-2i}{4k}-\frac{j}{k})\gamma}\boxtimes V^{T_2,\pm}=V^{T_1,+}\oplus V^{T_1,-},  \  {\rm if} \ \frac{k}{2}-i \ {\rm is \ odd}.
$$
From  \cite{JW2} , as $K_0^{\sigma_2}$-modules,  
for $k\in 2\Z$, $i\in 2\mathbb{Z}+1$, 
\begin{equation}\label{e44}
	\begin{split}
&M^{i,i^{'}}\boxtimes W^{+}=M^{i,i^{'}}\boxtimes W^{-}=\sum\limits_{\tiny{\begin{split}|i-\frac{k}{2}|\leq l<\frac{k}{2} \\ i+\frac{k}{2}+l\in 2\mathbb{Z} \ \  \\ i+l\leq \frac{3k}{2}\ \ \ \end{split}}}  \Big(W(k,l)^{+}+W(k,l)^{-}\Big),
\end{split}
\end{equation}
where $W=W(k,\frac{k}{2})$ or $W=\widetilde{W(k,\frac{k}{2})}$. Then by Lemma \ref{l3.16}, for $i\in 2\Z+1$, 
$$\sum\limits_{\tiny{\begin{split}0\leqslant l\leqslant \frac{k}{2}-1 \\  i+j+l\in 2\mathbb{Z} \ \ \ \end{split}}} (\overline{L(k,l)}^{\sigma_3, +}\oplus\overline{L(k,l)}^{\sigma_3, -})\subseteq 
\overline{L(k,i)}^{\sigma_1,\pm}\boxtimes \overline{L(k,\frac{k}{2})}^{\sigma_2,j}.$$
Since the quantum dimensions of both sides are equal,  we have (\ref{e3.40}).  For (\ref{e3.45})-(\ref{e3.46}), by Theorem 5.3 of \cite{JW2},  
if $k\in 2\Z$, $i\in 2\mathbb{Z}$, $i^{'}\in 2\mathbb{Z}+1$, 
\begin{eqnarray}\begin{split}
&M^{i,i^{'}}\boxtimes W^{+}=M^{i,i^{'}}\boxtimes W^{-}\\
&=\sum\limits_{\tiny{\begin{split}|i-\frac{k}{2}|\leq l< \frac{k}{2} \\  i+\frac{k}{2}+l\in 2\mathbb{Z} \ \  \\ i+l\leq \frac{3k}{2}\  \ \ \end{split}}}  \Big(W(k,l)^{+}+W(k,l)^{-}\Big)+\Big(\bar{W}^{+}+\bar{W}^{-}\Big),
\label{eq:5.29.}\end{split}
\end{eqnarray}
where if $W=W(k,\frac{k}{2})$, then $\bar{W}=\widetilde{W(k,\frac{k}{2})}$, and if $W=\widetilde{W(k,\frac{k}{2})}$, then $\bar{W}=W(k,\frac{k}{2})$.

If $k\in 2\Z$, $i\in 2\mathbb{Z}$, $i^{'}\in 2\mathbb{Z}$, we have
\begin{eqnarray}\begin{split}
&M^{i,i^{'}}\boxtimes W^{+}=M^{i,i^{'}}\boxtimes W^{-}\\
&=\sum\limits_{\tiny{\begin{split}|i-\frac{k}{2}|\leq l< \frac{k}{2} \\  i+\frac{k}{2}+l\in 2\mathbb{Z} \ \  \\ i+l\leq \frac{3k}{2} \ \ \
		\end{split}}} \Big(W(k,l)^{+}+W(k,l)^{-}\Big)+\Big(W^{+}+W^{-}\Big),
\label{eq:5.30.}\end{split}
\end{eqnarray}
where $W=W(k,\frac{k}{2})$ or $W=\widetilde{W(k,\frac{k}{2})}$. Then by (\ref{e3.58}), (\ref{e3.60}), and (\ref{e3.41})-(\ref{e3.43}), we have for $i\in 2\Z$, $i\neq \frac{k}{2}$, 
$$
\sum\limits_{\tiny{\begin{split}0\leqslant l\leqslant \frac{k}{2}-1 \\   i+\frac{k}{2}+l\in 2\mathbb{Z} \ \ \ \end{split}}} \overline{L(k,l)}^{\sigma_3}\oplus (\overline{L(k,\frac{k}{2})}^{\sigma_3, 1}\oplus\overline{L(k,\frac{k}{2})}^{\sigma_3, 4})\subseteq \overline{L(k,i)}^{\sigma_1,+}\boxtimes \overline{L(k,\frac{k}{2})}^{\sigma_2,1},
$$
and 
$$
\sum\limits_{\tiny{\begin{split}0\leqslant l\leqslant \frac{k}{2}-1 \\   i+\frac{k}{2}+l\in 2\mathbb{Z} \ \ \ \end{split}}} \overline{L(k,l)}^{\sigma_3}\oplus (\overline{L(k,\frac{k}{2})}^{\sigma_3, 2}\oplus\overline{L(k,\frac{k}{2})}^{\sigma_3, 3})\subseteq \overline{L(k,i)}^{\sigma_1,-}\boxtimes \overline{L(k,\frac{k}{2})}^{\sigma_2,1}.
$$
Since the  quantum dimensions of both sides of the above two relations are equal,
 we get (\ref{e3.45}) and (\ref{e3.46}). (\ref{e3.47})-(\ref{e3.48}) follow from Lemma \ref{l3.8} and (\ref{e1})-(\ref{e3}).

\end{proof}
\begin{thm}\label{t3.16}  For  $k\in 4\Z_++2$, we have
\begin{eqnarray}\begin{split}
&\overline{L(k,\frac{k}{2})}^{\sigma_1, i}\boxtimes \overline{L(k,\frac{k}{2})}^{\sigma_2, j}=\overline{L(k,\frac{k}{2})}^{\sigma_1, r}\boxtimes \overline{L(k,\frac{k}{2})}^{\sigma_2, s} 
 \\
&
=\bigoplus_{0\leq l<\frac{k}{2}, l\in 2\Z}\overline{L(k,l)}^{\sigma_3, +}, \quad i,j=1,4, \ r,s=2,3; \label{e3.72}\end{split}
\end{eqnarray}
\vskip 0.1cm
\begin{eqnarray}\begin{split}
&\overline{L(k,\frac{k}{2})}^{\sigma_1, i}\boxtimes \overline{L(k,\frac{k}{2})}^{\sigma_2, j}=\overline{L(k,\frac{k}{2})}^{\sigma_1, j}\boxtimes \overline{L(k,\frac{k}{2})}^{\sigma_2, i}\\&
=\bigoplus_{0\leq l<\frac{k}{2}, l\in 2\Z}\overline{L(k,l)}^{\sigma_3, -}, \quad i=1,4, \ j=2,3; \label{e3.73}
\end{split}
\end{eqnarray}
\vskip 0.1cm
\begin{eqnarray}\begin{split}
&\overline{L(k,\frac{k}{2})}^{\sigma_1, i}\boxtimes \overline{L(k,\frac{k}{2})}^{\sigma_3, j}=\overline{L(k,\frac{k}{2})}^{\sigma_1, r}\boxtimes \overline{L(k,\frac{k}{2})}^{\sigma_3, s} 
\\
&
=\bigoplus_{0\leq l<\frac{k}{2}, l\in 2\Z}\overline{L(k,l)}^{\sigma_2, +}, \quad i,j=1,3, \ r,s=2,4; \  \label{e4}\end{split}
\end{eqnarray}
\vskip 0.1cm
\begin{eqnarray}\begin{split}
&\overline{L(k,\frac{k}{2})}^{\sigma_1, i}\boxtimes \overline{L(k,\frac{k}{2})}^{\sigma_3, j}=\overline{L(k,\frac{k}{2})}^{\sigma_1, j}\boxtimes \overline{L(k,\frac{k}{2})}^{\sigma_3, i} 
\\
&
=\bigoplus_{0\leq l<\frac{k}{2}, l\in 2\Z}\overline{L(k,l)}^{\sigma_2, -}, \quad i=2,4,  \ j=1,3; \  \label{e5}\end{split}
\end{eqnarray}
\vskip 0.1cm
\begin{eqnarray}\begin{split}
&\overline{L(k,\frac{k}{2})}^{\sigma_2, i}\boxtimes \overline{L(k,\frac{k}{2})}^{\sigma_3, j}=\overline{L(k,\frac{k}{2})}^{\sigma_2, r}\boxtimes \overline{L(k,\frac{k}{2})}^{\sigma_3, s} 
\\
&
=\bigoplus_{0\leq l<\frac{k}{2}, l\in 2\Z}\overline{L(k,l)}^{\sigma_1, +}, \quad i=1,3,  \ j=1,4; \ r=2,4, \ s=2,3; \label{e5}\end{split}
\end{eqnarray}
\vskip 0.1cm
\begin{eqnarray}\begin{split}
&\overline{L(k,\frac{k}{2})}^{\sigma_2, i}\boxtimes \overline{L(k,\frac{k}{2})}^{\sigma_3, j}=\overline{L(k,\frac{k}{2})}^{\sigma_2, r}\boxtimes \overline{L(k,\frac{k}{2})}^{\sigma_3, s} 
\\
&
=\bigoplus_{0\leq l<\frac{k}{2}, l\in 2\Z}\overline{L(k,l)}^{\sigma_1, -}, \quad i=1,3,  \ j=2,3; \ r=2,4, \ s=1,4; \label{e6}\end{split}
\end{eqnarray}
\end{thm}

\begin{proof}  
	 We first prove (\ref{e3.72}) for $j=1$. Notice that 
	 $$
 V_{\Z\gamma}^+\otimes (M^{\frac{k}{2},0})^+
\oplus V_{\Z\gamma}^-\otimes (M^{\frac{k}{2},0})^-\subseteq	\overline{L(k,\frac{k}{2})}^{\sigma_1,1},
	$$
and by (\ref{e3.60}), we have for $k\in 4\Z_++2$, 
	$$	\overline{L(k,\frac{k}{2})}^{\sigma_2,1}=V^{T_2,+}\otimes W(k,\frac{k}{2})^+\oplus V^{T_2,-}\otimes W(k,\frac{k}{2})^-.$$
By   (\ref{e11}) and (\ref{e12}), we have for $k\in 4\Z+2$,
\begin{eqnarray*}
(M^{\frac{k}{2},0})^{+}\boxtimes W^{\pm}=\bigoplus_{0\leq i< \frac{k}{2},l\in 2\mathbb{Z}}W(k,l)^{\pm},
\end{eqnarray*}
\begin{eqnarray*}
(M^{\frac{k}{2},0})^{-}\boxtimes W^{\pm}=\bigoplus_{0\leq i< \frac{k}{2},l\in 2\mathbb{Z}}W(k,l)^{\mp},
\end{eqnarray*}
where $W=W(k,\frac{k}{2})$ or $W=\widetilde{W(k,\frac{k}{2})}$. 
Then by  fusion rules of $V_{\Z\gamma}^+$ \cite{Ab01} and Lemma \ref{l3.16}, we have
$$
\bigoplus_{0\leq l<\frac{k}{2}, l\in 2\Z}\overline{L(k,l)}^{\sigma_3, +}\subseteq \overline{L(k,\frac{k}{2})}^{\sigma_1, 1}\boxtimes \overline{L(k,\frac{k}{2})}^{\sigma_2, 1}.
$$
Recall from Proposition \ref{prop3.1} that 
	$$
{\rm dim}_q\overline{L(k,i)}^{\sigma_r,+}={\rm dim}_q\overline{L(k,i)}^{\sigma_r,-}=2\dfrac{\sin\frac{(i+1)\pi}{k+2}}{\sin \frac{\pi}{k+2}}, \ i\neq\frac{k}{2}, \ r=1,2,3,
$$	
$$
{\rm dim}_qL(k, \frac{k}{2})^{\sigma_r,j}=\dfrac{1}{\sin \frac{\pi}{k+2}}, \ r=1,2,3, 	\  j=1,2,3,4. 
$$	Since for $0\leq i\leq k$, 
$$
\dfrac{\sin\frac{(i+1)\pi}{k+2}}{\sin \frac{\pi}{k+2}}=\dfrac{\sin\frac{(k-i+1)\pi}{k+2}}{\sin \frac{\pi}{k+2}},
$$
we have 
$$
\begin{array}{ll}
&\bigoplus_{0\leq l<\frac{k}{2}, l\in 2\Z}{\rm dim}_q \overline{L(k,l)}^{\sigma_3, +}
=\bigoplus_{0\leq l<\frac{k}{2}, l\in 2\Z}2\dfrac{\sin\frac{(l+1)\pi}{k+2}}{\sin \frac{\pi}{k+2}}\\&\\
&=\bigoplus_{0\leq l\leq k, l\in 2\Z}\dfrac{\sin\frac{(l+1)\pi}{k+2}}{\sin \frac{\pi}{k+2}}
=(\dfrac{1}{\sin \frac{\pi}{k+2}})^2\\&\\
&={\rm dim}_q (L(k, \frac{k}{2})^{\sigma_1,1}\boxtimes L(k, \frac{k}{2})^{\sigma_2,1}).
\end{array}
$$
Then we deduce that
$$
\overline{L(k,\frac{k}{2})}^{\sigma_1, 1}\boxtimes \overline{L(k,\frac{k}{2})}^{\sigma_2, 1}=
\bigoplus_{0\leq l<\frac{k}{2}, l\in 2\Z}\overline{L(k,l)}^{\sigma_3, +},
$$
which is (\ref{e3.72}) for $j=1$. For the other cases,  notice that by (\ref{e1})-(\ref{e3}) and Lemma \ref{l3.8}, 
$$
\overline{L(k,\frac{k}{2})}^{\sigma_2,4}=L(k,0)^{\sigma_2,4}\boxtimes \overline{L(k,\frac{k}{2})}^{\sigma_2,1}, 
$$
$$
L(k,0)^{\sigma_3,2}\boxtimes \overline{L(k,i)}^{\sigma_3,\pm}= \overline{L(k,i)}^{\sigma_3,\pm}, 
\ 
L(k,0)^{\sigma_2,4}=L(k,0)^{\sigma_3,2}. 
$$
Then we have 
$$
\begin{array}{ll}
&\overline{L(k,\frac{k}{2})}^{\sigma_1, 1}\boxtimes \overline{L(k,\frac{k}{2})}^{\sigma_2, 4}
=\overline{L(k,\frac{k}{2})}^{\sigma_1, 1}\boxtimes L(k,0)^{\sigma_2,4}\boxtimes \overline{L(k,\frac{k}{2})}^{\sigma_2, 1}\\&\\
&=
\bigoplus_{0\leq l<\frac{k}{2}, l\in 2\Z} L(k,0)^{\sigma_3,2}\boxtimes\overline{L(k,l)}^{\sigma_3, +}
=\bigoplus_{0\leq l<\frac{k}{2}, l\in 2\Z}\overline{L(k,l)}^{\sigma_3, +}.
\end{array}
$$
Similarly, by (\ref{e1})-(\ref{e3}), Lemma \ref{lem3.7}  and Lemma \ref{l3.8}, 
$$
\begin{array}{ll}
&\overline{L(k,\frac{k}{2})}^{\sigma_1, 2}\boxtimes \overline{L(k,\frac{k}{2})}^{\sigma_2, 3}={L(k, 0)}^{\sigma_1, 2}\boxtimes {L(k, 0)}^{\sigma_2, 3}\boxtimes \overline{L(k,\frac{k}{2})}^{\sigma_1, 1}\boxtimes\overline{L(k,\frac{k}{2})}^{\sigma_2, 1}\\&\\
&={L(k, 0)}^2\boxtimes {L(k, 0)}^2\boxtimes \overline{L(k,\frac{k}{2})}^{\sigma_1, 1}\boxtimes\overline{L(k,\frac{k}{2})}^{\sigma_2, 1}\\&\\
&
={L(k, 0)}^1\boxtimes \overline{L(k,\frac{k}{2})}^{\sigma_1, 1}\boxtimes\overline{L(k,\frac{k}{2})}^{\sigma_2, 1}=\bigoplus_{0\leq l<\frac{k}{2}, l\in 2\Z}\overline{L(k,l)}^{\sigma_3, +};
\end{array}
$$
\vskip 0.2cm
$$
\begin{array}{ll}
& \overline{L(k,\frac{k}{2})}^{\sigma_1, 1}\boxtimes \overline{L(k,\frac{k}{2})}^{\sigma_2, 2}=L(k,0)^3\boxtimes \overline{L(k,\frac{k}{2})}^{\sigma_1, 1}\boxtimes \overline{L(k,\frac{k}{2})}^{\sigma_2, 1}\\&\\
&=\bigoplus_{0\leq l<\frac{k}{2}, l\in 2\Z}L(k,0)^3\boxtimes \overline{L(k,l)}^{\sigma_3, +}=\bigoplus_{0\leq l<\frac{k}{2}, l\in 2\Z}L(k,0)^{\sigma_3,3}\boxtimes \overline{L(k,l)}^{\sigma_3, +}\\&\\
&=\bigoplus_{0\leq l<\frac{k}{2}, l\in 2\Z}\overline{L(k,l)}^{\sigma_3, -};
\end{array} 
$$
\vskip 0.2cm
$$
\begin{array}{ll}
& \overline{L(k,\frac{k}{2})}^{\sigma_1, 1}\boxtimes \overline{L(k,\frac{k}{2})}^{\sigma_2, 3}=L(k,0)^2\boxtimes \overline{L(k,\frac{k}{2})}^{\sigma_1, 1}\boxtimes \overline{L(k,\frac{k}{2})}^{\sigma_2, 1}\\&\\
&=\bigoplus_{0\leq l<\frac{k}{2}, l\in 2\Z}L(k,0)^2\boxtimes \overline{L(k,l)}^{\sigma_3, +}=\bigoplus_{0\leq l<\frac{k}{2}, l\in 2\Z}L(k,0)^{\sigma_3,4}\boxtimes \overline{L(k,l)}^{\sigma_3, +}\\&\\
&=\bigoplus_{0\leq l<\frac{k}{2}, l\in 2\Z}\overline{L(k,l)}^{\sigma_3, -};
\end{array} 
$$
\vskip 0.2cm
$$
\begin{array}{ll}
& \overline{L(k,\frac{k}{2})}^{\sigma_1, 2}\boxtimes \overline{L(k,\frac{k}{2})}^{\sigma_2, 4}=L(k,0)^2\boxtimes L(k,0)^{\sigma_2,4}\boxtimes\overline{L(k,\frac{k}{2})}^{\sigma_1, 1}\boxtimes \overline{L(k,\frac{k}{2})}^{\sigma_2, 1}\\&\\
&=L(k,0)^2\boxtimes L(k,0)^{4}\overline{L(k,\frac{k}{2})}^{\sigma_1, 1}\boxtimes \overline{L(k,\frac{k}{2})}^{\sigma_2, 1}=L(k,0)^3\boxtimes\overline{L(k,\frac{k}{2})}^{\sigma_1, 1}\boxtimes \overline{L(k,\frac{k}{2})}^{\sigma_2, 1}\\&\\
&=\bigoplus_{0\leq l<\frac{k}{2}, l\in 2\Z}L(k,0)^3\boxtimes \overline{L(k,l)}^{\sigma_3, +}=\bigoplus_{0\leq l<\frac{k}{2}, l\in 2\Z}L(k,0)^{\sigma_3,3}\boxtimes \overline{L(k,l)}^{\sigma_3, +}\\&\\
&=\bigoplus_{0\leq l<\frac{k}{2}, l\in 2\Z}\overline{L(k,l)}^{\sigma_3, -};
\end{array} 
$$
\vskip 0.2cm
$$
\begin{array}{ll}
& \overline{L(k,\frac{k}{2})}^{\sigma_1, 3}\boxtimes \overline{L(k,\frac{k}{2})}^{\sigma_2, 4}=L(k,0)^3\boxtimes L(k,0)^{\sigma_2,4}\boxtimes\overline{L(k,\frac{k}{2})}^{\sigma_1, 1}\boxtimes \overline{L(k,\frac{k}{2})}^{\sigma_2, 1}\\&\\
&=L(k,0)^3\boxtimes L(k,0)^{4}\boxtimes\overline{L(k,\frac{k}{2})}^{\sigma_1, 1}\boxtimes \overline{L(k,\frac{k}{2})}^{\sigma_2, 1}=L(k,0)^2\boxtimes\overline{L(k,\frac{k}{2})}^{\sigma_1, 1}\boxtimes \overline{L(k,\frac{k}{2})}^{\sigma_2, 1}\\&\\
&=\bigoplus_{0\leq l<\frac{k}{2}, l\in 2\Z}L(k,0)^2\boxtimes \overline{L(k,l)}^{\sigma_3, +}=\bigoplus_{0\leq l<\frac{k}{2}, l\in 2\Z}L(k,0)^{\sigma_3,4}\boxtimes \overline{L(k,l)}^{\sigma_3, +}\\&\\
&=\bigoplus_{0\leq l<\frac{k}{2}, l\in 2\Z}\overline{L(k,l)}^{\sigma_3, -};
\end{array} 
$$
\vskip 0.2cm
$$
\begin{array}{ll}
& \overline{L(k,\frac{k}{2})}^{\sigma_1, 2}\boxtimes \overline{L(k,\frac{k}{2})}^{\sigma_2, 1}=L(k,0)^2\boxtimes\overline{L(k,\frac{k}{2})}^{\sigma_1, 1}\boxtimes \overline{L(k,\frac{k}{2})}^{\sigma_2, 1}\\&\\
&=\bigoplus_{0\leq l<\frac{k}{2}, l\in 2\Z}L(k,0)^2\boxtimes \overline{L(k,l)}^{\sigma_3, +}=\bigoplus_{0\leq l<\frac{k}{2}, l\in 2\Z}L(k,0)^{\sigma_3,4}\boxtimes \overline{L(k,l)}^{\sigma_3, +}\\&\\&
=\bigoplus_{0\leq l<\frac{k}{2}, l\in 2\Z}\overline{L(k,l)}^{\sigma_3, -};
\end{array} 
$$
\vskip 0.2cm
$$
\begin{array}{ll}
& \overline{L(k,\frac{k}{2})}^{\sigma_1, 3}\boxtimes \overline{L(k,\frac{k}{2})}^{\sigma_2, 1}=L(k,0)^3\boxtimes\overline{L(k,\frac{k}{2})}^{\sigma_1, 1}\boxtimes \overline{L(k,\frac{k}{2})}^{\sigma_2, 1}\\&\\&=\bigoplus_{0\leq l<\frac{k}{2}, l\in 2\Z}L(k,0)^{\sigma_3,3}\boxtimes \overline{L(k,l)}^{\sigma_3, +}=\bigoplus_{0\leq l<\frac{k}{1}, l\in 2\Z}\overline{L(k,l)}^{\sigma_3, -};
\end{array} 
$$
\vskip 0.2cm
$$
\begin{array}{ll}
& \overline{L(k,\frac{k}{2})}^{\sigma_1, 4}\boxtimes \overline{L(k,\frac{k}{2})}^{\sigma_2, 1}=L(k,0)^4\boxtimes \overline{L(k,\frac{k}{2})}^{\sigma_1, 1}\boxtimes \overline{L(k,\frac{k}{2})}^{\sigma_2, 1}\\&\\&=L(k,0)^{\sigma_2,4}\boxtimes \overline{L(k,\frac{k}{2})}^{\sigma_1, 1}\boxtimes \overline{L(k,\frac{k}{2})}^{\sigma_2, 1}=
\overline{L(k,\frac{k}{2})}^{\sigma_1, 1}\boxtimes \overline{L(k,\frac{k}{2})}^{\sigma_2, 4};
\end{array} 
$$
\vskip 0.2cm
$$
\begin{array}{ll}
& \overline{L(k,\frac{k}{2})}^{\sigma_1, 3}\boxtimes \overline{L(k,\frac{k}{2})}^{\sigma_2, 2}=L(k,0)^3\boxtimes L(k,0)^3\boxtimes \overline{L(k,\frac{k}{2})}^{\sigma_1, 1}\boxtimes \overline{L(k,\frac{k}{2})}^{\sigma_2, 1}\\&\\&=L(k,0)^{1}\boxtimes \overline{L(k,\frac{k}{2})}^{\sigma_1, 1}\boxtimes \overline{L(k,\frac{k}{2})}^{\sigma_2, 1}=\overline{L(k,\frac{k}{2})}^{\sigma_1, 1}\boxtimes \overline{L(k,\frac{k}{2})}^{\sigma_2, 1};
\end{array} 
$$
\vskip 0.2cm
$$
\begin{array}{ll}
& \overline{L(k,\frac{k}{2})}^{\sigma_1, 4}\boxtimes \overline{L(k,\frac{k}{2})}^{\sigma_2, 3}=L(k,0)^4\boxtimes L(k,0)^2\boxtimes \overline{L(k,\frac{k}{2})}^{\sigma_1, 1}\boxtimes \overline{L(k,\frac{k}{2})}^{\sigma_2, 1}\\&\\&=L(k,0)^{3}\boxtimes \overline{L(k,\frac{k}{2})}^{\sigma_1, 1}\boxtimes \overline{L(k,\frac{k}{2})}^{\sigma_2, 1}=L(k,0)^{\sigma_3,3}\boxtimes \overline{L(k,\frac{k}{2})}^{\sigma_1, 1}\boxtimes \overline{L(k,\frac{k}{2})}^{\sigma_2, 1}\\&\\&=\bigoplus_{0\leq l<\frac{k}{1}, l\in 2\Z}\overline{L(k,l)}^{\sigma_3, -};
\end{array} 
$$
\vskip 0.2cm
$$
\begin{array}{ll}
& \overline{L(k,\frac{k}{2})}^{\sigma_1, 2}\boxtimes \overline{L(k,\frac{k}{2})}^{\sigma_2, 2}=L(k,0)^2\boxtimes L(k,0)^3\boxtimes \overline{L(k,\frac{k}{2})}^{\sigma_1, 1}\boxtimes \overline{L(k,\frac{k}{2})}^{\sigma_2, 1}\\&\\&=L(k,0)^{4}\boxtimes \overline{L(k,\frac{k}{2})}^{\sigma_1, 1}\boxtimes \overline{L(k,\frac{k}{2})}^{\sigma_2, 1}=L(k,0)^{\sigma_3,2}\boxtimes \overline{L(k,\frac{k}{2})}^{\sigma_1, 1}\boxtimes \overline{L(k,\frac{k}{2})}^{\sigma_2, 1}\\&\\&=\bigoplus_{0\leq l<\frac{k}{2}, l\in 2\Z}\overline{L(k,l)}^{\sigma_3, +};
\end{array} 
$$
\vskip 0.2cm
$$
\begin{array}{ll}
& \overline{L(k,\frac{k}{2})}^{\sigma_1, 3}\boxtimes \overline{L(k,\frac{k}{2})}^{\sigma_2, 3}=L(k,0)^3\boxtimes L(k,0)^2\boxtimes \overline{L(k,\frac{k}{2})}^{\sigma_1, 1}\boxtimes \overline{L(k,\frac{k}{2})}^{\sigma_2, 1}\\&\\&=L(k,0)^{4}\boxtimes \overline{L(k,\frac{k}{2})}^{\sigma_1, 1}\boxtimes \overline{L(k,\frac{k}{2})}^{\sigma_2, 1}=L(k,0)^{\sigma_3,2}\boxtimes \overline{L(k,\frac{k}{2})}^{\sigma_1, 1}\boxtimes \overline{L(k,\frac{k}{2})}^{\sigma_2, 1}\\&\\&=\bigoplus_{0\leq l<\frac{k}{2}, l\in 2\Z}\overline{L(k,l)}^{\sigma_3, +};
\end{array} 
$$
\vskip 0.2cm
$$
\begin{array}{ll}
& \overline{L(k,\frac{k}{2})}^{\sigma_1, 4}\boxtimes \overline{L(k,\frac{k}{2})}^{\sigma_2, 4}=L(k,0)^4\boxtimes L(k,0)^4\boxtimes \overline{L(k,\frac{k}{2})}^{\sigma_1, 1}\boxtimes \overline{L(k,\frac{k}{2})}^{\sigma_2, 1}\\&\\&=L(k,0)^{1}\boxtimes \overline{L(k,\frac{k}{2})}^{\sigma_1, 1}\boxtimes \overline{L(k,\frac{k}{2})}^{\sigma_2, 1}
=\overline{L(k,\frac{k}{2})}^{\sigma_1, 1}\boxtimes \overline{L(k,\frac{k}{2})}^{\sigma_2, 1}.
\end{array} 
$$
The proof for (\ref{e3.73})-(\ref{e6}) is similar.

\end{proof}
\begin{thm}\label{t3.17}
	For $k\in 4\Z_+$, we have
	\begin{eqnarray}\begin{split}
	&\overline{L(k,\frac{k}{2})}^{\sigma_1, 1}\boxtimes \overline{L(k,\frac{k}{2})}^{\sigma_2, 1}=\overline{L(k,\frac{k}{2})}^{\sigma_1, 2}\boxtimes \overline{L(k,\frac{k}{2})}^{\sigma_2, 3}
	=\overline{L(k,\frac{k}{2})}^{\sigma_1, 3}\boxtimes \overline{L(k,\frac{k}{2})}^{\sigma_2, 2}\\&\\&=\overline{L(k,\frac{k}{2})}^{\sigma_1, 4}\boxtimes \overline{L(k,\frac{k}{2})}^{\sigma_2, 4}=\overline{L(k,\frac{k}{2})}^{\sigma_3,1} \bigoplus_{0\leq l<\frac{k}{2}, l\in 2\Z}\overline{L(k,l)}^{\sigma_3, +}, \label{e7}  
	\end{split}
	\end{eqnarray}
	\vskip 0.1cm
	\begin{eqnarray}\begin{split}
	&\overline{L(k,\frac{k}{2})}^{\sigma_1, 1}\boxtimes \overline{L(k,\frac{k}{2})}^{\sigma_2, 2}=\overline{L(k,\frac{k}{2})}^{\sigma_1, 2}\boxtimes \overline{L(k,\frac{k}{2})}^{\sigma_2, 4}
=\overline{L(k,\frac{k}{2})}^{\sigma_1, 3}\boxtimes \overline{L(k,\frac{k}{2})}^{\sigma_2, 1}\\&\\&=\overline{L(k,\frac{k}{2})}^{\sigma_1, 4}\boxtimes \overline{L(k,\frac{k}{2})}^{\sigma_2, 3}	=\overline{L(k,\frac{k}{2})}^{\sigma_3,3} \bigoplus_{0\leq l<\frac{k}{2}, l\in 2\Z}\overline{L(k,l)}^{\sigma_3, -},\label{e8}   
\end{split}
	\end{eqnarray}
	\vskip 0.1cm
	\begin{eqnarray}\begin{split}
&	\overline{L(k,\frac{k}{2})}^{\sigma_1, 1}\boxtimes \overline{L(k,\frac{k}{2})}^{\sigma_2, 3}=\overline{L(k,\frac{k}{2})}^{\sigma_1, 2}\boxtimes \overline{L(k,\frac{k}{2})}^{\sigma_2, 1}
	=	\overline{L(k,\frac{k}{2})}^{\sigma_1, 3}\boxtimes \overline{L(k,\frac{k}{2})}^{\sigma_2, 4}\\&\\&=\overline{L(k,\frac{k}{2})}^{\sigma_1, 4}\boxtimes \overline{L(k,\frac{k}{2})}^{\sigma_2, 2}=\overline{L(k,\frac{k}{2})}^{\sigma_3,4} \bigoplus_{0\leq l<\frac{k}{2}, l\in 2\Z}\overline{L(k,l)}^{\sigma_3, -}, \label{e9}  
	\end{split}
	\end{eqnarray}
	\vskip 0.1cm
	\begin{eqnarray}\begin{split}\label{e3.78}
&	\overline{L(k,\frac{k}{2})}^{\sigma_1, 1}\boxtimes \overline{L(k,\frac{k}{2})}^{\sigma_2, 4}=\overline{L(k,\frac{k}{2})}^{\sigma_1, 4}\boxtimes \overline{L(k,\frac{k}{2})}^{\sigma_2, 1}
	=\overline{L(k,\frac{k}{2})}^{\sigma_1, 2}\boxtimes \overline{L(k,\frac{k}{2})}^{\sigma_2, 2}\\&\\&=\overline{L(k,\frac{k}{2})}^{\sigma_1, 3}\boxtimes \overline{L(k,\frac{k}{2})}^{\sigma_2, 3}=\overline{L(k,\frac{k}{2})}^{\sigma_3,2} \bigoplus_{0\leq l<\frac{k}{2
		}, l\in 2\Z}\overline{L(k,l)}^{\sigma_3, +}.  \label{e10}
	\end{split}
	\end{eqnarray}
	
\end{thm}
\begin{proof} We only prove  (\ref{e7}), since  the proof for other formulas is similar.
	By (\ref{e13}) and (\ref{e14}), we have for $k\in 4\Z$, 
	$$
		(M^{\frac{k}{2},0})^{+}\boxtimes W(k,\frac{k}{2})^{+}=\sum\limits_{\tiny{\begin{split}0\leq l\leq \frac{k}{2}-1 \\ l\in 2\mathbb{Z} \ \ \ \ \ \end{split}}} W(k,l)^{+}+W(k,\frac{k}{2})^{+},
	$$
	$$
		(M^{\frac{k}{2},0})^{-}\boxtimes W(k,\frac{k}{2})^{+}=\sum\limits_{\tiny{\begin{split}0\leq l\leq \frac{k}{2}-1 \\ l\in 2\mathbb{Z} \ \ \ \ \ \end{split}}} W(k,l)^{-}+W(k,\frac{k}{2})^{-}.
	$$
 Notice that 
$$
V_{\Z\gamma}^+\otimes (M^{\frac{k}{2},0})^+
\oplus V_{\Z\gamma}^-\otimes (M^{\frac{k}{2},0})^-\subseteq	\overline{L(k,\frac{k}{2})}^{\sigma_1,1},
$$
and by (\ref{e3.41}) and (\ref{e3.59}), 
$$	\overline{L(k,\frac{k}{2})}^{\sigma_2,1}=V^{T_1,+}\otimes W(k,\frac{k}{2})^+\oplus V^{T_1,-}\otimes W(k,\frac{k}{2})^-.$$
Then by (\ref{e3.41}) we have 
$$
\overline{L(k,\frac{k}{2})}^{\sigma_3,1} \bigoplus_{0\leq l<\frac{k}{2}, l\in 2\Z}\overline{L(k,l)}^{\sigma_3, +}\subseteq \overline{L(k,\frac{k}{2})}^{\sigma_1, 1}\boxtimes \overline{L(k,\frac{k}{2})}^{\sigma_2, 1}.$$	
Since the quantum dimensions of both sides are equal, we deduce that
$$
\overline{L(k,\frac{k}{2})}^{\sigma_1, 1}\boxtimes \overline{L(k,\frac{k}{2})}^{\sigma_2, 1}=\overline{L(k,\frac{k}{2})}^{\sigma_3,1} \bigoplus_{0\leq l<\frac{k}{2}, l\in 2\Z}\overline{L(k,l)}^{\sigma_3, +}.
$$
For other cases of (\ref{e7}),  by (\ref{e1})-(\ref{e3}), Lemma \ref{lem3.7}  and Lemma \ref{l3.8},  we have
$$
\begin{array}{ll}
& \overline{L(k,\frac{k}{2})}^{\sigma_1, 2}\boxtimes \overline{L(k,\frac{k}{2})}^{\sigma_2, 3}\\&\\&
=L(k,0)^{2}\boxtimes L(k,0)^{\sigma_2,3} \boxtimes  \overline{L(k,\frac{k}{2})}^{\sigma_1, 1}\boxtimes \overline{L(k,\frac{k}{2})}^{\sigma_2, 1}\\&\\&
=L(k,0)^{2}\boxtimes L(k,0)^{2} \boxtimes  \overline{L(k,\frac{k}{2})}^{\sigma_1, 1}\boxtimes \overline{L(k,\frac{k}{2})}^{\sigma_2, 1}\\&\\&
=L(k,0)^{1} \boxtimes  \overline{L(k,\frac{k}{2})}^{\sigma_1, 1}\boxtimes \overline{L(k,\frac{k}{2})}^{\sigma_2, 1}=\overline{L(k,\frac{k}{2})}^{\sigma_1, 1}\boxtimes\overline{L(k,\frac{k}{2})}^{\sigma_2, 1};
\end{array}
$$
\vskip 0.2cm
$$\begin{array}{ll}
	& \overline{L(k,\frac{k}{2})}^{\sigma_1, 3}\boxtimes \overline{L(k,\frac{k}{2})}^{\sigma_2, 2}\\&\\&
	=L(k,0)^{3}\boxtimes L(k,0)^{\sigma_2,2} \boxtimes  \overline{L(k,\frac{k}{2})}^{\sigma_1, 1}\boxtimes \overline{L(k,\frac{k}{2})}^{\sigma_2, 1}\\&\\&
	=L(k,0)^{3}\boxtimes L(k,0)^{3} \boxtimes  \overline{L(k,\frac{k}{2})}^{\sigma_1, 1}\boxtimes \overline{L(k,\frac{k}{2})}^{\sigma_2, 1}\\&\\&
	=L(k,0)^{1} \boxtimes  \overline{L(k,\frac{k}{2})}^{\sigma_1, 1}\boxtimes \overline{L(k,\frac{k}{2})}^{\sigma_2, 1}=\overline{L(k,\frac{k}{2})}^{\sigma_1, 1}\boxtimes\overline{L(k,\frac{k}{2})}^{\sigma_2, 1};
\end{array}
$$
\vskip 0.1cm
$$\begin{array}{ll}
& \overline{L(k,\frac{k}{2})}^{\sigma_1, 4}\boxtimes \overline{L(k,\frac{k}{2})}^{\sigma_2, 4}\\&\\&
=L(k,0)^{4}\boxtimes L(k,0)^{\sigma_2,4} \boxtimes  \overline{L(k,\frac{k}{2})}^{\sigma_1, 1}\boxtimes \overline{L(k,\frac{k}{2})}^{\sigma_2, 1}\\&\\&
=L(k,0)^{4}\boxtimes L(k,0)^{4} \boxtimes  \overline{L(k,\frac{k}{2})}^{\sigma_1, 1}\boxtimes \overline{L(k,\frac{k}{2})}^{\sigma_2, 1}\\&\\&
=L(k,0)^{1} \boxtimes  \overline{L(k,\frac{k}{2})}^{\sigma_1, 1}\boxtimes \overline{L(k,\frac{k}{2})}^{\sigma_2, 1}=\overline{L(k,\frac{k}{2})}^{\sigma_1, 1}\boxtimes\overline{L(k,\frac{k}{2})}^{\sigma_2, 1}.
\end{array}
$$
\end{proof}

\end{document}